\documentclass[12pt]{amsart}
\usepackage{amsmath}
\usepackage{amsxtra}
\usepackage{amscd}
\usepackage{amsthm}
\usepackage{amsfonts}
\usepackage{amssymb}
\usepackage{eucal}
\usepackage{scalerel}
\usepackage{verbatim}
\usepackage[matrix,arrow,curve]{xy}

\usepackage[pdftex,bookmarks=false,colorlinks=true,debug=true, naturalnames=true,pdfnewwindow=true]{hyperref}

\newtheorem{cor}[subsection]{Corollary}

\newtheorem{lem}[subsection]{Lemma}
\newtheorem{prop}[subsection]{Proposition}

\newtheorem{conj}[subsection]{Conjecture}
\newtheorem{thm}[subsection]{Theorem}
\newtheorem{defn}[subsection]{Definition}

\theoremstyle{definition}

\theoremstyle{remark}

\newcommand{\nc}{\newcommand}
\nc{\renc}{\renewcommand} \nc{\ssec}{\subsection}
\nc{\sssec}{\subsubsection}

\nc{\on}{\operatorname} \nc{\wh}{\widehat}
\nc\ol{\overline} \nc\ul{\underline} \nc\wt{\widetilde}

\nc{\BA}{{\mathbb{A}}} \nc{\BC}{{\mathbb{C}}}
\nc{\BD}{{\mathbb{D}}} \nc{\BG}{{\mathbb{G}}} \nc{\BQ}{{\mathbb{Q}}}
\nc{\BM}{{\mathbb{M}}} \nc{\BN}{{\mathbb{N}}}
\nc{\BP}{{\mathbb{P}}} \nc{\BR}{{\mathbb{R}}}
\nc{\BZ}{{\mathbb{Z}}} \nc{\BS}{{\mathbb{S}}} \nc{\BW}{{\mathbb{W}}}

\nc{\CA}{{\mathcal{A}}} \nc{\CB}{{\mathcal{B}}} \nc{\CalD}{{\mathcal{D}}}
\nc{\CE}{{\mathcal{E}}} \nc{\CF}{{\mathcal{F}}}
\nc{\CG}{{\mathcal{G}}} \nc{\CH}{{\mathcal{H}}}
\nc{\CI}{{\mathcal{I}}} \nc{\CK}{{\mathcal{K}}} \nc{\CL}{{\mathcal{L}}}
\nc{\CM}{{\mathcal{M}}} \nc{\CN}{{\mathcal{N}}}
\nc{\CO}{{\mathcal{O}}} \nc{\CP}{{\mathcal{P}}}
\nc{\CQ}{{\mathcal{Q}}} \nc{\CR}{{\mathcal{R}}}
\nc{\CS}{{\mathcal{S}}} \nc{\CT}{{\mathcal{T}}}
\nc{\CU}{{\mathcal{U}}} \nc{\CV}{{\mathcal{V}}}  \nc{\CY}{{\mathcal Y}}
\nc{\CW}{{\mathcal{W}}} \nc{\CZ}{{\mathcal{Z}}}

\nc{\cM}{{\check{\mathcal M}}{}} \nc{\csM}{{\check{\mathcal A}}{}}
\nc{\oM}{{\overset{\circ}{\mathcal M}}{}}
\nc{\obM}{{\overset{\circ}{\mathbf M}}{}}
\nc{\oCA}{{\overset{\circ}{\mathcal A}}{}}
\nc{\obA}{{\overset{\circ}{\mathbf A}}{}}
\nc{\ooM}{{\overset{\circ}{M}}{}}
\nc{\osM}{{\overset{\circ}{\mathsf M}}{}}
\nc{\vM}{{\overset{\bullet}{\mathcal M}}{}}
\nc{\nM}{{\underset{\bullet}{\mathcal M}}{}}
\nc{\oD}{{\overset{\circ}{\mathcal D}}{}}
\nc{\obD}{{\overset{\circ}{\mathbf D}}{}}
\nc{\oA}{{\overset{\circ}{\mathbb A}}{}}
\nc{\op}{{\overset{\bullet}{\mathbf p}}{}}
\nc{\cp}{{\overset{\circ}{\mathbf p}}{}}
\nc{\oU}{{\overset{\bullet}{\mathcal U}}{}}
\nc{\ofZ}{{\overset{\circ}{\mathfrak Z}}{}}

\nc{\ff}{{\mathfrak{f}}} \nc{\fv}{{\mathfrak{v}}}
\nc{\fa}{{\mathfrak{a}}} \nc{\fb}{{\mathfrak{b}}}
\nc{\fd}{{\mathfrak{d}}} \nc{\fe}{{\mathfrak{e}}}
\nc{\fg}{{\mathfrak{g}}} \nc{\fgl}{{\mathfrak{gl}}}
\nc{\fh}{{\mathfrak{h}}} \nc{\fri}{{\mathfrak{i}}}
\nc{\fj}{{\mathfrak{j}}} \nc{\fk}{{\mathfrak{k}}} \nc{\fl}{{\mathfrak{l}}}
\nc{\fm}{{\mathfrak{m}}} \nc{\fn}{{\mathfrak{n}}}
\nc{\ft}{{\mathfrak{t}}} \nc{\fu}{{\mathfrak{u}}}
\nc{\fw}{{\mathfrak{w}}} \nc{\fz}{{\mathfrak{z}}}
\nc{\fp}{{\mathfrak{p}}} \nc{\frr}{{\mathfrak{r}}}
\nc{\fs}{{\mathfrak{s}}} \nc{\fsl}{{\mathfrak{sl}}}
\nc{\hsl}{{\widehat{\mathfrak{sl}}}}
\nc{\hgl}{{\widehat{\mathfrak{gl}}}}
\nc{\hg}{{\widehat{\mathfrak{g}}}}
\nc{\chg}{{\widehat{\mathfrak{g}}}{}^\vee}
\nc{\hn}{{\widehat{\mathfrak{n}}}}
\nc{\chn}{{\widehat{\mathfrak{n}}}{}^\vee}

\nc{\fA}{{\mathfrak{A}}} \nc{\fB}{{\mathfrak{B}}} \nc{\fC}{{\mathfrak{C}}}
\nc{\fD}{{\mathfrak{D}}} \nc{\fE}{{\mathfrak{E}}}
\nc{\fF}{{\mathfrak{F}}} \nc{\fG}{{\mathfrak{G}}} \nc{\fH}{{\mathfrak{H}}}
\nc{\fI}{{\mathfrak{I}}} \nc{\fJ}{{\mathfrak{J}}}
\nc{\fK}{{\mathfrak{K}}} \nc{\fL}{{\mathfrak{L}}}
\nc{\fM}{{\mathfrak{M}}} \nc{\fN}{{\mathfrak{N}}}
\nc{\frP}{{\mathfrak{P}}} \nc{\fQ}{{\mathfrak{Q}}}
\nc{\fT}{{\mathfrak{T}}} \nc{\fU}{{\mathfrak{U}}}
\nc{\fV}{{\mathfrak{V}}} \nc{\fW}{{\mathfrak{W}}}
\nc{\fX}{{\mathfrak{X}}} \nc{\fY}{{\mathfrak{Y}}}
\nc{\fZ}{{\mathfrak{Z}}}

\nc{\ba}{{\mathbf{a}}}
\nc{\bb}{{\mathbf{b}}} \nc{\bc}{{\mathbf{c}}}
\nc{\be}{{\mathbf{e}}} \nc{\bj}{{\mathbf{j}}}
\nc{\bn}{{\mathbf{n}}} \nc{\bp}{{\mathbf{p}}}
\nc{\bq}{{\mathbf{q}}} \nc{\br}{{\mathbf{r}}} \nc{\bt}{{\mathbf{t}}}
\nc{\bfu}{{\mathbf{u}}} \nc{\bv}{{\mathbf{v}}}
\nc{\bx}{{\mathbf{x}}} \nc{\by}{{\mathbf{y}}}
\nc{\bw}{{\mathbf{w}}} \nc{\bA}{{\mathbf{A}}}
\nc{\bB}{{\mathbf{B}}} \nc{\bC}{{\mathbf{C}}}
\nc{\bD}{{\mathbf{D}}} \nc{\bF}{{\mathbf{F}}}
\nc{\bH}{{\mathbf{H}}} \nc{\bJ}{{\mathbf{J}}} \nc{\bK}{{\mathbf{K}}}
\nc{\bM}{{\mathbf{M}}} \nc{\bN}{{\mathbf{N}}}
\nc{\bO}{{\mathbf{O}}} \nc{\bS}{{\mathbf{S}}} \nc{\bT}{{\mathbf{T}}}
\nc{\bV}{{\mathbf{V}}} \nc{\bW}{{\mathbf{W}}}
\nc{\bX}{{\mathbf{X}}}
\nc{\bY}{{\mathbf{Y}}} \nc{\bP}{{\mathbf{P}}}
\nc{\bZ}{{\mathbf{Z}}} \nc{\bh}{{\mathbf{h}}}

\nc{\sA}{{\mathsf{A}}} \nc{\sB}{{\mathsf{B}}}
\nc{\sC}{{\mathsf{C}}} \nc{\sD}{{\mathsf{D}}}
\nc{\sE}{{\mathsf{E}}} \nc{\sF}{{\mathsf{F}}} \nc{\sG}{{\mathsf{G}}}
\nc{\sI}{{\mathsf{I}}} \nc{\sK}{{\mathsf{K}}} \nc{\sL}{{\mathsf{L}}}
\nc{\sfm}{{\mathsf{m}}} \nc{\sM}{{\mathsf{M}}} \nc{\sO}{{\mathsf{O}}}
\nc{\sQ}{{\mathsf{Q}}} \nc{\sP}{{\mathsf{P}}}
\nc{\sT}{{\mathsf{T}}} \nc{\sZ}{{\mathsf{Z}}}
\nc{\sV}{{\mathsf{V}}} \nc{\sW}{{\mathsf{W}}}
\nc{\sfp}{{\mathsf{p}}} \nc{\sr}{{\mathsf{r}}}
\nc{\st}{{\mathsf{t}}} \nc{\sfb}{{\mathsf{b}}}
\nc{\sfc}{{\mathsf{c}}} \nc{\sd}{{\mathsf{d}}}
\nc{\sz}{{\mathsf{z}}}

\nc{\tA}{{\widetilde{\mathbf{A}}}}
\nc{\tB}{{\widetilde{\mathcal{B}}}}
\nc{\tg}{{\widetilde{\mathfrak{g}}}} \nc{\tG}{{\widetilde{G}}}
\nc{\TM}{{\widetilde{\mathbb{M}}}{}}
\nc{\tO}{{\widetilde{\mathsf{O}}}{}}
\nc{\tU}{{\widetilde{\mathfrak{U}}}{}} \nc{\TZ}{{\tilde{Z}}}
\nc{\tx}{{\tilde{x}}} \nc{\tbv}{{\tilde{\bv}}}
\nc{\tfP}{{\widetilde{\mathfrak{P}}}{}} \nc{\tz}{{\tilde{\zeta}}}
\nc{\tmu}{{\tilde{\mu}}}

\nc{\urho}{\underline{\rho}} \nc{\uB}{\underline{B}}
\nc{\uC}{{\underline{\mathbb{C}}}} \nc{\ui}{\underline{i}}
\nc{\uj}{\underline{j}} \nc{\ofP}{{\overline{\mathfrak{P}}}}
\nc{\oB}{{\overline{\mathcal{B}}}}
\nc{\og}{{\overline{\mathfrak{g}}}} \nc{\oI}{{\overline{I}}}

\nc{\eps}{\varepsilon} \nc{\hrho}{{\hat{\rho}}}
\nc{\blambda}{{\bar\lambda}} \nc{\bmu}{{\bar\mu}} \nc{\bnu}{{\bar\nu}}

\nc{\one}{{\mathbf{1}}} \nc{\two}{{\mathbf{t}}}

\nc{\Sym}{{\mathop{\operatorname{\rm Sym}}}}
\nc{\Tot}{{\mathop{\operatorname{\rm Tot}}}}
\nc{\Spec}{{\mathop{\operatorname{\rm Spec}}}}
\nc{\Ker}{{\mathop{\operatorname{\rm Ker}}}}
\nc{\Hilb}{{\mathop{\operatorname{\rm Hilb}}}}
\nc{\End}{{\mathop{\operatorname{\rm End}}}}
\nc{\Ext}{{\mathop{\operatorname{\rm Ext}}}}
\nc{\Hom}{{\mathop{\operatorname{\rm Hom}}}}
\nc{\CHom}{{\mathop{\operatorname{{\mathcal{H}}\it om}}}}
\nc{\GL}{{\mathop{\operatorname{\rm GL}}}}
\nc{\gr}{{\mathop{\operatorname{\rm gr}}}}
\nc{\Id}{{\mathop{\operatorname{\rm Id}}}}
\nc{\defi}{{\mathop{\operatorname{\rm def}}}}
\nc{\length}{{\mathop{\operatorname{\rm length}}}}
\nc{\supp}{{\mathop{\operatorname{\rm supp}}}}
\nc{\HC}{{\mathcal H}{\mathcal C}}

\nc{\Cliff}{{\mathsf{Cliff}}}
\nc{\Fl}{{\mathcal{F}\ell}} \nc{\Fib}{{\mathsf{Fib}}}
\nc{\Coh}{{\mathsf{Coh}}} \nc{\FCoh}{{\mathsf{FCoh}}}

\nc{\reg}{{\text{\rm reg}}}
\nc{\gvee}{{\mathfrak g}^{\!\scriptscriptstyle\vee}}
\nc{\tvee}{{\mathfrak t}^{\!\scriptscriptstyle\vee}}
\nc{\nvee}{{\mathfrak n}^{\!\scriptscriptstyle\vee}}
\nc{\bvee}{{\mathfrak b}^{\!\scriptscriptstyle\vee}}
       \nc{\rhovee}{\rho^{\!\scriptscriptstyle\vee}}

\nc{\cplus}{{\mathbf{C}_+}} \nc{\cminus}{{\mathbf{C}_-}}
\nc{\cthree}{{\mathbf{C}_*}} \nc{\Qbar}{{\bar{Q}}}

\newcommand{\oZ}{\vphantom{j^{X^2}}\smash{\overset{\circ}{\vphantom{\rule{0pt}{0.55em}}\smash{Z}}}}
\newcommand\iso{\,\vphantom{j^{X^2}}\smash{\overset{\sim}{\vphantom{\rule{0pt}{0.20em}}\smash{\longrightarrow}}}\,}
\nc{\Gtimes}{\vphantom{j^{X^2}}\smash{\overset{G}{\vphantom{\rule{0pt}{0.30em}}\smash{\times}}}}
\nc{\sGtimes}{\vphantom{j^{X^2}}\smash{\overset{\mathsf G}{\vphantom{\rule{0pt}{0.30em}}\smash{\times}}}}

\nc{\bOmega}{{\overline{\Omega}}}

\nc{\seq}[1]{\stackrel{#1}{\sim}}

\nc{\aff}{{\operatorname{aff}}}
\nc{\fin}{{\operatorname{fin}}}
\nc{\Gr}{{\operatorname{Gr}}}
\nc{\GR}{{\mathbf{Gr}}}
\nc{\Perv}{{\operatorname{Perv}}}
\nc{\Rep}{{\operatorname{Rep}}}
\nc{\IC}{{\operatorname{IC}}}
\nc{\Bun}{{\operatorname{Bun}}}
\nc{\Proj}{{\operatorname{Proj}}}
\nc{\pt}{{\operatorname{pt}}}
\nc{\bfmu}{{\boldsymbol{\mu}}}
\nc{\bfomega}{{\boldsymbol{\omega}}}
\nc{\calM}{\mathcal M}
\nc{\calA}{\mathcal A}
\nc{\calO}{\mathcal O}
\nc{\CC}{\mathbb C}
\nc{\calN}{\mathcal N}
\nc{\grg}{\mathfrak g}
\nc{\tslash}{/\!\!/\!\!/}
\nc\grt{\mathfrak t}
\nc\bfM{\mathbf M}
\nc\bfN{\mathbf N}
\nc\Sig{\Sigma}
\nc\ZZ{\mathbb{Z}}
\nc\calC{\mathcal C}
\nc\calF{\mathcal F}
\nc\calX{\mathcal X}
\nc\calY{\mathcal Y}
\nc\QCoh{\operatorname{QCoh}}
\nc\IndCoh{\operatorname{IndCoh}}
\nc\Maps{\operatorname{Maps}}
\nc\Dmod{D-\operatorname{mod}}
\newcommand\Hecke{\operatorname{Hecke}}
\nc{\calD}{\mathcal D}
\nc\bfO{\mathbf O}
\renewcommand{\AA}{\mathbb A}
\nc\GG{\mathbb G}
\nc\calK{\mathcal K}
\nc{\calG}{\mathcal G}
\nc\RHom{\operatorname{RHom}}
\nc\grs{\mathfrak s}
\nc{\tilX}{\widetilde X}
\nc\calB{\mathcal B}
\nc\calS{\mathcal S}
\nc\calT{\mathcal T}
\nc\calZ{\mathcal Z}
\nc\LS{\operatorname{LocSys}}
\nc\bfL{\on{\mathbf L}}
\nc\Sing{\on{Sing}}
\nc\x{\times}
\nc\calW{\mathcal W}

\begin{document}
\author[A.~Braverman]{Alexander Braverman}
\address{
Department of Mathematics,
University of Toronto and Perimeter Institute of Theoretical Physics,
Waterloo, Ontario, Canada, N2L 2Y5
\newline
Skolkovo Institute of Science and Technology;
}
\email{braval@math.toronto.edu}
\author[M.~Finkelberg]{Michael Finkelberg}
\address{
National Research University Higher School of Economics, Russian Federation\newline
Department of Mathematics, 6 Usacheva st., Moscow 119048;\newline
Skolkovo Institute of Science and Technology;\newline
Institute for Information Transmission Problems}
\email{fnklberg@gmail.com}

\title
%[Dva Idiota]
[Coulomb branches of 3-dimensional gauge theories]
{Coulomb branches of 3-dimensional gauge theories and related structures}

%\dedicatory{To Borya Feigin on his 60th birthday}

%\thanks{{\bf Mathematics Subject Classification (2000).}
%19E08, (22E65, 37K10).}

%\thanks{{\bf Key words.} $q$-difference Toda lattice, Equivariant
%$K$-theory, Laumon compactification.}

%\thanks{The work of L.R. was partially supported
%by  RFBR grants 07-01-92214-CNRSL-a and 05-01-02805-CNRSL-a.
%L.R. gratefully acknowledges the support from Deligne 2004 Balzan
%prize in mathematics.}

\begin{abstract}
These are (somewhat informal) lecture notes for the CIME summer school "Geometric Representation Theory and Gauge Theory​"
in June 2018.
In these notes we review the constructions and results of \cite{bfn2,bfn3,bfn4} where a
mathematical definition of Coulomb branches of 3d N=4 quantum gauge theories (of cotangent
type) is given, and also present a framework for studying some further mathematical structures
(e.g.\ categories of line operators in the corresponding topologically twisted theories)
related to these theories.
\end{abstract}
\maketitle

\section{Introduction and first motivation: Symplectic duality and a little bit of physics}

\subsection{Symplectic singularities}
Let $X$ be an algebraic variety over $\CC$. We say that $X$ is singular symplectic (or $X$ has symplectic singularities) if

(1) $X$ is a normal Poisson variety;

(2) There exists a smooth dense open subset $U$ of $X$ on which the Poisson structure comes from a symplectic structure. We shall denote by $\omega$ the corresponding symplectic form.

(3) There exists a resolution of singularities $\pi:\tilX\to X$ such that $\pi^*\omega$ has no poles on $\tilX$.

\noindent
This definition is due to A.~Beauville, who showed that if condition (3) above holds for some $\tilX$ then it holds for any resolution of $X$
\subsection{Conical symplectic singularities}
We say that $X$ is a conical symplectic singularity if in addition to (1)-(3) above the following conditions are satisfied:

(4) $X$ is affine;

(5) There exists a $\CC^{\times}$-action on $X$ which contracts it to a point $x_0\in X$ and such that the form $\omega$ has positive weight.

\noindent
We shall consider examples a little later.

\subsection{Symplectic resolutions}
By a symplectic resolutions we mean a morphism $\pi:\tilX\to X$ such that

(a) $X$ satisfies (1)-(5) above;

(b) $\tilX$ is smooth and $\pi$ is proper and birational and the action of $\CC^{\times}$ on $X$ extends to an action on
$\tilX$.

(c) $\pi^*\omega$ extends to a symplectic form on $\tilX$.

\noindent
{\bf Example.} Let $\grg$ be a semi-simple Lie algebra over $\CC$ and let $\calN_{\grg}\subset \grg^*$ be its nilpotent cone. Let $\calB$ denote the flag variety of $\grg$. Then the Springer map $\pi:T^*\calB\to \calN_{\grg}$ is proper and birational, so if we let $X=\calN_{\grg}, \tilX=T^*\calB$ we get a symplectic resolution.

\subsection{The spaces $\grt_X$ and $\grs_X$}
To any conical symplectic singularity $X$ one can associate two canonical vector spaces which we shall denote by $\grt_X$ and $\grs_X$. The space $\grs_X$ is just the Cartan subalgebra of the group of Hamiltonian automorphisms of $X$ commuting with the contracting $\BC^\times$-action
(which is an finite-dimensional algebraic group over $\CC$). The space $\grt_X$ is trickier to define. First, assume that $X$ has a symplectic resolution $\tilX$. Then
$\grt_X=H^2(\tilX,\CC)$ (it follows from the results of Namikawa that $\grt_X$ is independent of the choice of $\tilX$).
Moreover, $\grt_X$ also has another interpretation: there is a deformation $\calX$ of $X$ as a singular symplectic variety over the base $\grt_X$. The map $\calX\to \grt_X$ is smooth away from a finite union of hyperplanes in $\grt_X$.

If $X$ doesn't have have a symplectic resolution, Namikawa still defines the space $\grt_X$ and the above deformation; the only difference is
that in this case $\calX$ is no longer smooth over the generic point of $\grt_X$ (informally, one can say that $\calX$ is the deformation which makes $X$ ``as smooth as possible" while staying in the class of symplectic varieties).

\subsection{Some examples}
Let $X$ be $\calN_{\grg}$ as in the example above. Then $\grs_X$ is the Cartan subalgebra of $\grg$ and $\grt_X$ is its dual space.
One may think that $\grt_X$ and $\grs_X$ always have the same dimension. However, it is not true already in the case $X=\CC^{2n}$.
In this case $\grs_X$ has dimension $n$ and $\grt_X=0$.

Let now $X$ be a Kleinian surface singularity of type $A,D$ or $E$. In other words $X$ is isomorphic to $\CC^2/\Gamma$ where $\Gamma$ is a finite subgroup of $SL(2,\CC)$. Thus $X$ has a unique singular point and it is known that $X$ has a symplectic resolution $\tilX$ with exceptional divisor formed by a tree of ${\mathbb P}^1$'s whose intersection matrix is the above Cartan matrix of type $A,D$ or $E$. Thus the dimension of $\grt_X$ is the rank of this Cartan matrix. On the other hand, it is easy to show that $\grs_X$ is 1-dimensional if $\Gamma$ is of type $A$ and is equal to $0$ otherwise.

\subsection{The idea of symplectic duality}
The idea of symplectic duality is this:\footnote{The main ideas are due to T.~Braden, A.~Licata, N.~Proudfoot and B.~Webster.} often conical symplectic singularities come in ``dual" pairs $(X,X^*)$ (the assignment $X\to X^*$ is by no means a functor; we just have a lot of interesting examples of dual pairs). What does it mean that $X$ and $X^*$ are dual?
There is no formal definition; however, there are a lot of interesting properties that a dual pair must satisfy. The most straightforward one is this: we should have
$$
\grt_X=\grs_{X^*}\quad \grs_X=\grt_{X^*}.
$$
Other properties of dual pairs are more difficult to describe. For example, if both $X$ and $X^*$ have symplectic resolutions $\tilX$ and
$\tilX^*$ then one should
have
$$
\dim H^*(\tilX,\CC)=\dim H^*(\tilX^*,\CC).
$$
(However, these spaces are not supposed to be canonically isomorphic).
We refer the reader to \cite{BPW}, \cite{BLPW} for more details.

One of the purposes of these notes will be to provide a construction of a large class of symlectically dual pairs. Before we discuss what this class is, let us talk about some examples.

\subsection{Examples of symplectically dual spaces}
\subsubsection{Nilpotent cones}Let
$X=\calN_{\grg}$ and let $X^*=\calN_{\grg^{\vee}}$ where $\grg^{\vee}$ is the Langlands dual Lie algebra. This is supposed to be a symplectically dual pair.

  \subsubsection{Slodowy slices in type $A$}
  For partitions $\lambda\geq\mu$ of $n$, let $\CS^\lambda_\mu$
  be the intersection of the nilpotent orbit closure
  $\ol{\mathbb O}_\lambda\subset\fgl(n)$ with the Slodowy slice to
  the orbit ${\mathbb O}_\mu$. Then $\CS^\lambda_\mu$ is dual to
  $\CS^{\mu^t}_{\lambda^t}$.

  \subsubsection{Toric hyperk\"ahler manifolds}
Consider an exact sequence
\begin{equation*}
0\to\BZ^{d-n}\xrightarrow{\alpha}\BZ^d\xrightarrow{\beta}\BZ^n\to0
\end{equation*}
of the free based $\BZ$-modules. It gives rise to a toric hyperk\"ahler manifold
$X$~\cite{bida}. Then $X^*$ is the toric hyperk\"ahler manifold associated to the
dual exact sequence (Gale duality).

  \subsubsection{Uhlenbeck spaces} $\on{Sym}^a(\BA^2/\Gamma)^\vee\simeq\CU_G^a(\BA^2)/\BG_a^2$
  for a finite subgroup $\Gamma\subset SL(2)$ corresponding by McKay
  to an almost simple simply connected simply laced Lie group $G$. Here $\BG_a^2$ acts on $\BA^2$ by
  translations, and hence it acts on $\CU_G^a(\BA^2)$
  by the transport of structure.

\subsection{3d N=4 quantum field theories and symplectic duality}
One source of dual pairs $(X,X^*)$ comes from quantum field theory. We discuss this in more detail in Section \ref{naive}; here we are just going to mention briefly the relevant notions.

Physicists have a notion of 3-dimensional N=4 super-symmetric quantum field theory. Any such theory $\calT$ is supposed to have a well-defined moduli space of vacua $\calM(\calT)$. This space is somewhat too complicated for our present discussion. Instead we are going to discuss some ``easy" parts of this space. Namely, the above moduli space of vacua should have two special pieces called the Higgs and the Coulomb branch of the moduli space of vacua; we shall denote these by $\calM_H(\calT)$ and $\calM_C(\calT)$. They are supposed to be Poisson (generically symplectic) complex algebraic varieties (in fact, the don't even have to be algebraic but for simplicity we shall only consider examples when they are).
They should also be hyper-k\"ahler in some sense, but (to the best of our knowledge) this notion is not well-defined for singular varieties, we are going to ignore the hyper-k\"ahler structure in these notes. But at least they are expected to be singular symplectic.

There is no mathematical classification of 3d N=4 theories. However, here is a class of examples. Let $G$ be a complex reductive algebraic group and let $\bfM$ be a symplectic representation of $G$; moreover we shall assume that the action of $G$ is Hamiltonian, i.e.\ that we have a well-defined moment map $\mu:\bfM\to \grg^*$ (this map can be fixed uniquely by requiring that $\mu(0)=0$). Then to the pair $(G,\bfM)$ one is supposed to associate
a theory $\calT(G,\bfM)$. This theory is called {\em gauge theory with gauge group $G$ and matter $\bfM$}. Its Higgs branch is expected to be equal to $\bfM\tslash G$ -- the Hamiltonian reduction of $\bfM$ with respect to $G$.
In particular, all Nakajima quiver varieties arise in this way (the corresponding theories are called {\em quiver gauge theories}).

What about the Coulomb branch of gauge theories? These are more tricky to define. Physicists have some expectations about those but no rigorous definition in general.
For example, $\calM_C(G,\bfM)$ is supposed to be birationally isomorphic to $(T^*T^{\vee})/W$. Here $T^{\vee}$ is the torus to dual the Cartan torus of $G$ and $W$ is the Weyl group. The above birational isomorphism should also preserve the Poisson structure.\footnote{$(T^*T^{\vee})/W$ is actually the Coulomb branch of the corresponding classical field theory and the fact that the above birational isomorphism is not in general biregular means that ``in the quantum theory the Coulomb branch acquires quantum corrections".}
In addition $\calM_C(G,\bfM)$ has a canonical $\CC^{\times}$-action with respect to which the symplectic for has weight 2.
Unfortunately,  it is not always conical but very often it is. Roughly speaking, to guarantee that $\calM_C(G,\bfM)$ is conical one needs that the representation $\bfM$ be ``big enough"
(for reasons not to be discussed here physicists call the corresponding gauge theories ``good or ugly").
In the conical case physicists (cf.\ \cite{chz}) produce a formula for the graded character of the algebra of functions on $\calM_C(G,\bfM)$.
This formula is called ``the monopole formula" (in a special case relevant for the purposes of these notes it is recalled in Subsection \ref{monopole}).

The idea is that at least in the conical case the pair $(\calM_H(\calT),\calM_C(\calT))$ should produce an example of a dual symplectic pair.
One of the purposes of these notes (but not the only purpose) is to review the contents of the papers
\cite{bfn2,bfn3,bfn4} (joint with H.~Nakajima) where a mathematical definition of the Coulomb branches $\calM_C(G,\bfM)$ ia given under an additional assumption (namely, we assume that $\bfM=T^*\bfN=\bfN\oplus\bfN^*$ for some representation $\bfN$ of $G$
-- such theories are called {\em gauge theories of cotangent type}) and some further properties of Coulomb brancnes are studied.\footnote{The reader is also advised to consult the papers \cite{na3,na4,na5} by H.~Nakajima. In particular, the papers \cite{bfn2,bfn3,bfn4} are based on the ideas developed earlier in \cite{na3}. Also \cite{na4} contains a lot of interesting open problems in the subject most of which will not be addressed in these notes.}
In this case we shall write $\calM_C(G,\bfN)$ instead of $\calM_C(G,\bfM)$. In {\em loc. cit.} it is defined as $\Spec(\CC[\calM_C(G,\bfN)])$ where $\CC[\calM_C(G,\bfN)]$ is some geometrically defined algebra over $\CC$. The varieties $\calM_C(G,\bfN)$ are normal, affine, Poisson and generically symplectic. We expect that the they are actually singular symplectic, but we can't prove this in general. The main ingredient in the definition is the geometry of the affine Grassmannian $\Gr_G$ of $G$.

\subsection{Remark about categorical symplectic duality} The following will never be used in the sequel, but we think it is important to mention it for the interested reader. Perhaps the most interesting aspect of symplectic duality is the {\em categorical symplectic duality} discussed in \cite{BLPW}.
Namely, in {\em loc. cit.} the authors explain that if both $X$ and $X^*$ have a symplectic resolution, then one can think about symplectic duality between them as Koszul duality between some version of category $\calO$ over the quantization of the algebras $\CC[X]$ and $\CC[X^*]$. In fact, it is explained in \cite{w} that a slightly weaker version of this statement can be formulated even when $X$ and $X^*$ do not have a symplectic resolution. This weaker statement is in fact proved in \cite{w} for $\calM_H(G,\bfN)$ and $\calM_C(G,\bfN)$ (where the author uses the definition of $\calM_C(G,\bfN)$ from \cite{bfn2}).

\subsection{The plan}These notes are organized as follows. First, as was mentioned above the main geometric player in our construction of $\calM_C(G,\bfN)$ is the affine Grassmannian $\Gr_G$ of $G$. In Sections \ref{satake} and \ref{derived-satake} we review some facts and constructions related to $\Gr_G$. Namely, in Section \ref{satake} we review the so called {\em geometric Satake equivalence}; in Section \ref{derived-satake} we discuss an upgrade this construction: the {\em derived
  geometric Satake equivalence}.
In Section \ref{naive} we discuss some general expectations about 3d N=4 theories and in Section \ref{Coulomb} we give a definition of the varieties $\calM_C(G,\bfN)$. Section \ref{quiver} is devoted to the example of quiver gauge theories; in particular,  for finite quivers of $ADE$ type we identify
the Coulomb branches with certain (generalized) slices in the affine Grassmannian of the corresponding group of $ADE$-type.
Finally, in Section \ref{categorical} we discuss some conjectural categorical structures related to the {\em topologically twisted version}
of gauge theories of cotangent type (we have learned the main ideas of this Section from T.~Dimofte, D.~Gaiotto, J.~Hilburn and P.~Yoo).

\subsection{Acknowledgements}
We are greatly indebted to our coauthor H.~Nakajima who taught us everything we know about Coulomb branches of 3d N=4 gauge theories and to the organizers of the CIME summer school ``Geometric Representation Theory and Gauge Theory​" in June 2018,  for which these notes were written.
In addition we would like to thank T.~Dimofte, D.~Gaiotto, J.~Hilburn and P.~Yoo for their patient explanations of various things (in particular, as was mentioned above the main idea of Section \ref{categorical} is due to them). Also, we are very grateful to R.~Bezrukavnikov, K.~Costello, D.~Gaitsgory and S.~Raskin for their help with many questions  which arose during the preparation of these notes.
M.F.\ was partially supported by the Russian Academic Excellence Project `5-100'.

\section{Geometric Satake}\label{satake}

\subsection{Overview}
\label{overview}
Let $\CO$ denote the formal power series ring $\BC[[z]]$, and let $\CK$ denote
its fraction field $\BC((z))$. Let $G$ be a reductive complex algebraic
group with a Borel and a Cartan subgroup $G\supset B\supset T$, and with
the Weyl group $W$ of $(G,T)$. Let $\Lambda$
be the coweight lattice, and let $\Lambda^+\subset\Lambda$ be the submonoid of
dominant coweights. Let also $\Lambda_+\subset\Lambda$ be the submonoid spanned
by the simple coroots $\alpha_i,\ i\in I$.
We denote by $G^\vee\supset T^\vee$ the Langlands dual
group, so that $\Lambda$ is the weight lattice of $G^\vee$.

The affine Grassmannian $\Gr_G=G_\CK/G_\CO$ is an ind-projective scheme,
the union $\bigsqcup_{\lambda\in\Lambda^+}\Gr_G^\lambda$ of $G_\CO$-orbits.
The closure of $\Gr_G^\lambda$ is a projective variety
$\ol\Gr{}^\lambda_G=\bigsqcup_{\mu\leq\lambda}\Gr^\mu_G$. The fixed point set
$\Gr^T_G$ is naturally identified with the coweight lattice $\Lambda$;
and $\mu\in\Lambda$ lies in $\Gr_G^\lambda$ iff $\mu\in W\lambda$.

One of the cornerstones of the Geometric Langlands Program initiated by
V.~Drinfeld is an equivalence $\BS$ of the tensor category $\Rep(G^\vee)$ and
the category $\Perv_{G_\CO}(\Gr_G)$ of $G_\CO$-equivariant perverse constructible
sheaves on $\Gr_G$ equipped with a natural monoidal convolution structure
$\star$ and a fiber functor $H^\bullet(\Gr_G,-)$~\cite{lu,gi,bd,mv}. It is a
categorification of the classical Satake isomorphism between
$K(\Rep(G^\vee))=\BC[T^\vee]^{W}$ and the spherical affine Hecke algebra of
$G$. The geometric Satake equivalence $\BS$ sends an irreducible $G^\vee$-module
$V^\lambda$ with highest weight $\lambda$ to the Goresky-MacPherson sheaf
$\IC(\ol\Gr{}^\lambda_G)$.

In order to construct a commutativity constraint for
$(\Perv_{G_\CO}(\Gr_G),\star)$, Beilinson and Drinfeld introduced a relative
version $\Gr_{G,BD}$ of the Grassmannian over the Ran space of a smooth curve
$X$, and a fusion monoidal structure $\Psi$ on $\Perv_{G_\CO}(\Gr_G)$ (isomorphic
to $\star$). One of the main discoveries of~\cite{mv} was a $\Lambda$-grading
of the fiber functor
$H^\bullet(\Gr_G,\CF)=\bigoplus_{\lambda\in\Lambda}\Phi_\lambda(\CF)$ by the
hyperbolic stalks at $T$-fixed points. For a $G^\vee$-module $V$, its weight
space $V_\lambda$ is canonically isomorphic to the hyperbolic stalk
$\Phi_\lambda(\BS V)$.

Various geometric structures of a perverse sheaf $\BS V$ reflect some fine
representation theoretic structures of $V$, such as Brylinski-Kostant filtration
and the action of dynamical Weyl group, see~\cite{gr}. One of the important
technical tools of studying $\Perv_{G_\CO}(\Gr_G)$ is the embedding
$\Gr_G\hookrightarrow\GR_G$ into Kashiwara infinite type scheme
$\GR_G=G_{\BC((z^{-1}))}/G_{\BC[z]}$~\cite{ka1,kt}. The quotient
$G_{\BC[[z^{-1}]]}\backslash\GR_G$ is the moduli stack $\Bun_G(\BP^1)$ of
$G$-bundles on the projective line $\BP^1$.
The $G_{\BC[[z^{-1}]]}$-orbits on $\GR_G$ are of finite codimension; they are also
numbered by the dominant coweights of $G$, and the image of an orbit
$\GR_G^\lambda$ in $\Bun_G(\BP^1)$ consists of $G$-bundles of isomorphism type
$\lambda$~\cite{gro}. The stratifications
$\Gr_G=\bigsqcup_{\lambda\in\Lambda^+}\Gr_G^\lambda$ and
$\GR_G=\bigsqcup_{\lambda\in\Lambda^+}\GR_G^\lambda$ are transversal, and their
intersections and various generalizations thereof will play an important role later on.

More precisely, we denote by $K_1$ the first congruence subgroup of
$G_{\BC[[z^{-1}]]}$: the kernel of the evaluation projection
$\on{ev}_\infty\colon G_{\BC[[z^{-1}]]}\twoheadrightarrow G$. The transversal
slice $\CW_\mu^\lambda$ (resp.\ $\ol\CW{}_\mu^\lambda$) is defined as the
intersection of $\Gr_G^\lambda$ (resp.\ $\ol\Gr{}_G^\lambda$) and
$K_1\cdot\mu$ in $\GR_G$. It is known that $\ol\CW{}_\mu^\lambda$ is nonempty
iff $\mu\leq\lambda$, and $\dim\ol\CW{}_\mu^\lambda$ is an affine irreducible
variety of dimension $\langle2\rho^{\!\scriptscriptstyle\vee},\lambda-\mu\rangle$.
Following an idea of I.~Mirkovi\'c,~\cite{kwy} proved that
$\ol\CW{}_\mu^\lambda=\bigsqcup_{\mu\leq\nu\leq\lambda}\CW_\mu^{\nu}$
is the decomposition of $\ol\CW{}_\mu^\lambda$ into symplectic leaves of a
natural Poisson structure.

\subsection{Hyperbolic stalks}
\label{hyperbolic}
Let $N$ denote the unipotent radical of the Borel $B$, and let $N_-$ stand for the
unipotent radical of the opposite Borel $B_-$. For a coweight $\nu\in\Lambda=\Gr_G^T$,
we denote by $S_\nu\subset\Gr_G$ (resp.\ $T_\nu\subset\Gr_G$) the orbit of $N(\CK)$
(resp.\ of $N_-(\CK)$. The intersections $S_\nu\cap\ol\Gr{}^\lambda_G$
(resp.\ $T_\nu\cap\ol\Gr{}^\lambda_G$) are the
{\em attractors} (resp.\ {\em repellents}) of $\BC^\times$ acting via its homomorphism
$2\rho$ to the Cartan torus $T\curvearrowright\ol\Gr{}^\lambda_G\colon
S_\nu\cap\ol\Gr{}^\lambda_G=\{x\in\ol\Gr{}^\lambda_G :
\lim\limits_{c\to0}2\rho(c)\cdot x=\nu\}$ and
$T_\nu\cap\ol\Gr{}^\lambda_G=\{x\in\ol\Gr{}^\lambda_G :
\lim\limits_{c\to\infty}2\rho(c)\cdot x=\nu\}$. Going to the limit
$\Gr_G=\lim\limits_{\lambda\in\Lambda^+}\ol\Gr{}^\lambda_G,\ S_\nu$ (resp.\ $T_\nu$)
is the attractor (resp.\ repellent) of $\nu$ in $\Gr_G$.
We denote by $r_{\nu,+}$ (resp.\ $r_{\nu,-}$) the locally closed embedding
$S_\nu\hookrightarrow\Gr_G$ (resp.\ $T_\nu\hookrightarrow\Gr_G$).
We also denote by $\iota_{\nu,+}$ (resp.\ $\iota_{\nu,-}$) the closed embedding
of the point $\nu$ into $S_\nu$ (resp.\ into $T_\nu$).
The following theorem is proved in~\cite{br,dg}.

\begin{thm}
There is a canonical isomorphism of functors
$\iota_{\nu,+}^*r_{\nu,+}^!\simeq\iota_{\nu,-}^!r_{\nu,-}^*\colon
D^b_{G_\CO}(\Gr_G)\to D^b(\on{Vect})$.
\end{thm}

\begin{defn}
For a sheaf $\CF\in D^b_{G_\CO}(\Gr_G)$ we define its hyperbolic stalk at $\nu$
as $\Phi_\nu(\CF):=\iota_{\nu,+}^*r_{\nu,+}^!\CF\simeq\iota_{\nu,-}^!r_{\nu,-}^*\CF$.
\end{defn}

The following dimension estimate due to~\cite{mv} is crucial for the geometric Satake.

\begin{lem}
\label{estimate}
\textup{(a)} $S_\nu\cap\Gr_G^\lambda\ne\varnothing$ iff
$T_\nu\cap\Gr_G^\lambda\ne\varnothing$ iff $\nu$ has nonzero multiplicity
in the irreducible $G^\vee$-module $V^\lambda$ with highest weight $\lambda$.
This is also equivalent to $\nu\in\ol\Gr{}_G^\lambda$.

\textup{(b)} The nonempty intersection $S_\nu\cap\ol\Gr{}_G^\lambda$ is
equidimensional of dimension $\langle\nu+\lambda,\rho^{\!\scriptscriptstyle\vee}\rangle$.

\textup{(c)} The nonempty intersection $T_\nu\cap\ol\Gr{}_G^\lambda$ is
equidimensional of dimension $\langle\nu+w_0\lambda,\rho^{\!\scriptscriptstyle\vee}\rangle$.
Here $w_0$ is the longest element of the Weyl group $W$.
\end{lem}

\begin{cor}
\label{grading}
\textup{(a)} For $\CF\in\on{Perv}_{G_\CO}(\Gr_G)$, the hyperbolic stalk $\Phi_\nu(\CF)$
is concentrated in degree $\langle\nu,2\rho^{\!\scriptscriptstyle\vee}\rangle$.

\textup{(b)} There is a canonical direct sum decomposition $H^\bullet(\Gr_G,\CF)=
\bigoplus_{\nu\in\Lambda}\Phi_\nu(\CF)$.

\textup{(c)} The functor $H^\bullet(\Gr_G,-)\colon \on{Perv}_{G_\CO}(\Gr_G)\to
\on{Vect}^{\on{gr}}$, as well as its upgrade $\bigoplus_{\nu\in\Lambda}\Phi_\nu
\colon \on{Perv}_{G_\CO}(\Gr_G)\to\Rep(T^\vee)$, is exact and conservative.
\end{cor}

\subsection{Convolution}
\label{convo}
We have the following basic diagram:
\begin{equation}
  \label{convolut}
\Gr_G\times\Gr_G\xleftarrow{p}G_\CK\times\Gr_G\xrightarrow{q}
\Gr_G\wt\times\Gr_G\xrightarrow{m}\Gr_G.
\end{equation}
Here $\Gr_G\wt\times\Gr_G=G_\CK\stackrel{G_\CO}{\times}\Gr_G=
(G_\CK\times\Gr_G)/((g,x)\sim(gh^{-1},hx),\ h\in G_\CO)$.
Furthermore, $p$ is the projection on the first factor and identity on the second
factor, and the composition $m\circ q$ is the action morphism
$G_\CK\times\Gr_G\to\Gr_G$ (which clearly factors through
$G_\CK\stackrel{G_\CO}{\times}\Gr_G$).

\begin{defn}
Given $\CF_1,\CF_2\in D^b_{G_\CO}(\Gr_G)$, their convolution
$\CF_1\star\CF_2\in D^b_{G_\CO}(\Gr_G)$ is
defined as $\CF_1\star\CF_2:=m_*(\CF_1\wt\boxtimes\CF_2)$, where $\CF_1\wt\boxtimes\CF_2$ is the
descent of $p^*(\CF_1\boxtimes\CF_2)$, that is
$q^*(\CF_1\wt\boxtimes\CF)=p^*(\CF_1\boxtimes\CF_2)$.
\end{defn}

The next lemma is due to~\cite{lu,mv}. It follows from the
{\em stratified semismallness} of $m\colon \ol\Gr{}_G^{\lambda,\mu}:=
p^{-1}(\ol\Gr{}_G^\lambda)\stackrel{G_\CO}{\times}\ol\Gr{}_G^\mu\to
\ol\Gr{}_G^{\lambda+\mu}$. Here $\ol\Gr{}_G^{\lambda,\mu}$ is stratified
by the union of $\Gr_G^{\nu,\theta}:=
p^{-1}(\Gr_G^\nu)\stackrel{G_\CO}{\times}\Gr_G^\theta$ over
$\nu\leq\lambda,\ \theta\leq\mu$. The stratified semismallness in turn
follows from the dimension estimate of~Lemma~\ref{estimate}.

\begin{lem}
\label{exact}
Given $\CF_1,\CF_2\in\on{Perv}_{G_\CO}(\Gr_G)$, their convolution
$\CF_1\star\CF_2$ lies in $\on{Perv}_{G_\CO}(\Gr_G)$ as well.
\end{lem}

In order to define a commutativity constraint for $\star$, we will need
an equivalent construction of the monoidal structure on $\on{Perv}_{G_\CO}(\Gr_G)$
via {\em fusion} due to V.~Drinfeld.

\subsection{Fusion}
Let $X$ be a smooth curve, e.g.\ $X=\BA^1$. We have the following basic diagram:
\begin{equation*}
\begin{CD}
(\Gr_G\wt\times\Gr_G)_X @>i>> \Gr_{G,X}\wt\times\Gr_{G,X} @<j<< (\Gr_{G,X}\times\Gr_{G,X})|_U \\
@V{m_X}VV @V{m_{X^2}}VV @V{\wr}VV \\
\Gr_{G,X} @>i>> \Gr_{G,X^2} @<j<< (\Gr_{G,X}\times\Gr_{G,X})|_U \\
@V{\pi}VV @V{\pi}VV @V{\pi}VV \\
X @>\Delta>> X^2 @<j<< U.
\end{CD}
\end{equation*}

Here $U\hookrightarrow X^2$ is the open embedding of the complement to the
diagonal $\Delta_X\hookrightarrow X^2$.  Furthermore, for $n\in\BN,\ \Gr_{G,X^n}$
is the moduli space of the following data: $\{(x_1,\ldots,x_n)\in X^n,\ \CP_G,\ \tau\}$,
where $\CP_G$ is a $G$-bundle on $X$, and $\tau$ is a trivialization of $\CP_G$
on $X\setminus\{x_1,\ldots,x_n\}$. The projection $\pi\colon \Gr_{G,X^n}\to X^n$
forgets the data of $\CP_G$ and $\tau$. Note that $\Gr_{G,X^2}|_U\simeq
(\Gr_{G,X}\times\Gr_{G,X})|_U$, while $\Gr_{G,X^2}|_{\Delta_X}\simeq\Gr_{G,X}$.
Furthermore, $\Gr_{G,X}\wt\times\Gr_{G,X}$ is the moduli space of the following
data: $\{(x_1,x_2)\in X^2,\ \CP_G^1,\CP_G^2,\tau,\sigma\}$, where
$\CP_G^1,\CP_G^2$ are $G$-bundles on $X;\ \sigma\colon \CP_G^1|_{X\setminus x_2}\iso
\CP_G^2|_{X\setminus x_2}$, and $\tau$ is a trivialization of $\CP_G^1$ on $X\setminus x_1$.
Note that $(\Gr_{G,X}\wt\times\Gr_{G,X})|_U\simeq(\Gr_{G,X}\times\Gr_{G,X})|_U$,
while $(\Gr_{G,X}\wt\times\Gr_{G,X})|_{\Delta_X}$ is fibered over $X$ with fibers isomophic
to $\Gr_G\wt\times\Gr_G$. Finally, $m_{X^2}\colon \Gr_{G,X}\wt\times\Gr_{G,X}\to\Gr_{G,X^2}$
takes $(x_1,x_2,\CP_G^1,\CP_G^2,\tau,\sigma)$ to $(x_1,x_2,\CP_G^2,\tau')$ where
$\tau'=\sigma\circ\tau|_{X\setminus\{x_1,x_2\}}$. All the squares in the above diagram
are cartesian. The stratified semismallness property of the convolution morphism $m$
used in the proof of~Lemma~\ref{exact} implies the {\em stratified smallness} property
of the relative convolution morphism $m_{X^2}$.

Now given $\CF_1,\CF_2\in D^b_{G_\CO}(\Gr_G)$,
we can define the constructible complexes $\CF_{1,X},\CF_{2,X}$ on $\Gr_{G,X}$
smooth over $X$, and by descent  similarly to~Section~\ref{convo},
a constructible complex $\CF_{1,X}\wt\boxtimes\CF_{2,X}$ on $\Gr_{G,X}\wt\times\Gr_{G,X}$
smooth over $X^2$. Note that
$(\CF_{1,X}\wt\boxtimes\CF_{2,X})|_U=(\CF_{1,X}\boxtimes\CF_{2,X})|_U$.
 For simplicity, let us take $X=\BA^1$. Then by the proper base change
for nearby cycles $\Psi_{x_1-x_2}m_{X^2*}(\CF_{1,X}\wt\boxtimes\CF_{2,X})$ on
$\Gr_{G,X^2}$ we deduce $(\CF_1\star\CF_2)_X=
\Psi_{x_1-x_2}(\CF_{1,X}\boxtimes\CF_{2,X})|_U$. The RHS being manifestly symmetric,
we deduce the desired commutativity constraint for the convolution product $\star$.

Also, the above smoothness of $\CF_{1,X}\wt\boxtimes\CF_{2,X}$ over $X^2$
implies that $\pi_*m_{X^2*}(\CF_{1,X}\wt\boxtimes\CF_{2,X})$ is a constant sheaf
on $X^2$. Since its diagonal stalks are $H^\bullet(\Gr_G,\CF_1\star\CF_2)$, and
the off-diagonal stalks are $H^\bullet(\Gr_G,\CF_1)\otimes H^\bullet(\Gr_G,\CF_2)$,
we obtain that the cohomology functor $\on{Perv}_{G_\CO}(\Gr_G)\to\on{Vect}$
is a {\em tensor} functor: $H^\bullet(\Gr_G,\CF_1\star\CF_2)\iso
H^\bullet(\Gr_G,\CF_1)\otimes H^\bullet(\Gr_G,\CF_2)$.

\begin{cor}
The abelian category $\on{Perv}_{G_\CO}(\Gr_G)$ is equipped with a
symmetric monoidal structure $\star$ and a fiber functor
$H^\bullet(\Gr_G,-)$.
\end{cor}

By Tannakian formalism, the tensor category $\on{Perv}_{G_\CO}(\Gr_G)$ must
be equivalent to $\on{Rep}(G')$ for a proalgebraic group $G'$. It remains to
identify $G'$ with the Langlands dual group $G^\vee$. From the semisimplicity
of $\on{Perv}_{G_\CO}(\Gr_G)$, the group $G'$ must be reductive. The upgraded
fiber functor $\bigoplus_{\nu\in\Lambda}\Phi_\nu
\colon \on{Perv}_{G_\CO}(\Gr_G)\to\Rep(T^\vee)$ is tensor since the nearby cycles
commute with hyperbolic stalks by~\cite[Proposition~5.4.1.(2)]{na0}. Hence we
obtain a homomorphism $T^\vee\hookrightarrow G'$. Now it is easy to identify
$G'$ with $G^\vee$ using~Lemma~\ref{estimate}(a). We have proved

\begin{thm}
There is a tensor equivalence $\BS\colon \on{Rep}(G^\vee),\otimes\iso
\on{Perv}_{G_\CO}(\Gr_G),\star$.
\end{thm}

\section{Derived geometric Satake}\label{derived-satake}

In this section we extend the algebraic description of $\on{Perv}_{G_\CO}(\Gr_G)$ to
an algebraic description of the equivariant derived category
$D_{G_\CO\rtimes\BC^\times}(\Gr_G)$.

\subsection{Asymptotic Harish-Chandra bimodules}
\label{Harish}
First we develop the necessary algebraic machinery.
Let $U=U(\gvee)$ be the universal enveloping algebra, and let $U_\hbar$ be the
graded enveloping algebra, i.e.\ the graded $\BC[\hbar]$-algebra generated by
$\gvee$ with relations $xy-yx=\hbar[x,y]$ for $x,y\in\gvee$ (thus $U_\hbar$ is obtained
from $U$ by the Rees construction producing a graded algebra from the filtered one).
The adjoint action extends to the action of $G^\vee$ on $U_\hbar$.

We consider the category $\HC_\hbar$ of graded modules over
$U^2_\hbar:=U_\hbar\otimes_{\BC[\hbar]}U_\hbar\simeq U_\hbar\rtimes U$ equipped with an
action of $G^\vee$ (denoted $\beta\colon G^\vee\times M\to M$)
satisfying the following conditions:

\textup{(a)} The action $U^2_\hbar\otimes M\to M$ is $G^\vee$-equivariant;

\textup{(b)} for any $x\in\gvee$, the action of $x\otimes1+1\otimes x\in U^2_\hbar$
coincides with the action of $\hbar\cdot d\beta(x)$;

\textup{(c)} the module $M$ is finitely generated as a $U_\hbar\otimes1$-module
(equivalently, as a $1\otimes U_\hbar$-module).

The restriction from $U^2_\hbar$ to $U_\hbar\otimes1$ gives an equivalence of $\HC_\hbar$
with the category of $G^\vee$-modules equipped with a $G^\vee$-equivariant
$U_\hbar$-action.

\subsubsection{Example: free Harish-Chandra bimodules}
\label{free}
Let $V\in\on{Rep}(G^\vee)$. We define a free Harish-Chandra bimodule
$\on{Fr}(V)=U_\hbar\otimes V$ with the $G^\vee$-action
$g(y\otimes v)=\on{Ad}_g(y)\otimes g(v)$ and the $U^2_\hbar$-action
$(x\otimes u)(y\otimes v)=xyu\otimes v+\hbar xy\otimes u(v)$,
where $x,u\in\gvee\subset U_\hbar$. Thus, $\on{Fr}(V)$ is the induction of $V$
(the left adjoint functor to the restriction
$\on{res}\colon \HC_\hbar\to\on{Rep}(G^\vee)$). This is a projective object of the
category $\HC_\hbar$. We will denote by $\HC_\hbar^{\on{fr}}$ the full subcategory of
$\HC_\hbar$ formed by all the free Harish-Chandra bimodules.

\subsection{Kostant-Whittaker reduction}
\label{Whittaker}
We also consider the subalgebra
$U^2_\hbar(\nvee_-)=U_\hbar(\nvee_-)\rtimes U(\nvee_-)\subset U^2_\hbar$. We fix a regular
character $\psi\colon U_\hbar(\nvee_-)\to\BC[\hbar]$ taking value 1 at each generator
$f_i$. We extend it to a character
$\psi^{(2)}\colon U^2_\hbar(\nvee_-)=U_\hbar(\nvee_-)\rtimes U(\nvee_-)\to\BC[\hbar]$ trivial
on the second factor (its restriction to $1\otimes U_\hbar(\nvee_-)$ equals $-\psi$).

\begin{defn}
    \label{kw reduction}
For $M\in\HC_\hbar$ we set
$\varkappa_\hbar(M):=(M\overset{L}{\otimes}_{1\otimes U_\hbar(\nvee_-)}(-\psi))^{N^\vee_-}$
{\em (Kostant-Whittaker reduction)}.
It is equipped with an action of the Harish-Chandra center
$Z(U_\hbar)\otimes_{\BC[\hbar]}Z(U_\hbar)=\BC[\ft/W\times\ft/W\times\BA^1]$. Furthermore,
$\varkappa_\hbar(M)$ is graded by the action of the element $h$ from a principal
$\fs\fl_2=\langle e,h,f\rangle$-triple whose $e$ corresponds to $\psi$ under the Killing
form. All in all, $\varkappa_\hbar(M)$ is a $\BC^\times$-equivariant coherent sheaf on
$\ft/W\times\ft/W\times\BA^1$ (with respect to the natural $\BC^\times$-action on $\ft/W$).
\end{defn}

\subsubsection{Example: differential operators and quantum Toda lattice}
\label{Toda}
The ring of $\hbar$-differential operators on
$G^\vee,\ \CalD_\hbar(G^\vee)=U_\hbar\ltimes\BC[G^\vee]$ is an object of the Ind-completion
$\on{Ind}\HC_\hbar$. It carries one more commuting structure of a Harish-Chandra bimodule
(where $U_\hbar$ acts by the right-invariant $\hbar$-differential operators). We define
$\CK:=\varkappa_\hbar(\CalD_\hbar(G^\vee))$, an algebra in the category $\on{Ind}\HC_\hbar$.
It corepresents the functor $\Hom(M,\CK)=\varkappa_\hbar(DM)$ where
$DM=\Hom_{U_\hbar}(M,U_\hbar)$ is a duality on $\HC_\hbar$. Here $\Hom_{U_\hbar}$ is taken
with respect to the {\em right} action of $U_\hbar$, while the {\em left} actions of
$U_\hbar$ on $M$ and on itself are used to construct the left and right actions of
$U_\hbar$ on $\Hom_{U_\hbar}(M,U_\hbar)$. Finally, we define an associative algebra
$\CT_\hbar:=\varkappa_\hbar(\CK)
=\varkappa^{\on{right}}_\hbar\varkappa^{\on{left}}_\hbar\CalD_\hbar(G^\vee)$,
the quantum open Toda lattice.

\subsection{Deformation to the normal cone}
\label{deformation}
A well known construction associates to a closed subvariety $Z\subset X$ the deformation
to the normal cone $\CN_ZX$ projecting to $X\times\BA^1$; all the nonzero fibers are
isomorphic to $X$, while the zero fiber is isomorphic to the normal cone $C_ZX$.
Recall that $\CN_ZX$ is defined as the relative spectrum of the subsheaf of algebras
$\CO_X[\hbar^{\pm1}]$, generated over $\CO_{X\times\BA^1}$ by the elements of the following
type: $f\hbar^{-1}$, where $f\in\CO_X,\ f|_Z=0$.

Now if $M$ is a Harish-Chandra bimodule free over $\BC[\hbar]$, then the action of
$\BC[\ft/W\times\ft/W\times\BA^1]$ on $\varkappa_\hbar(M)$ extends uniquely to the action
of the ring of functions $\BC[\CN_\Delta(\ft/W\times\ft/W)]$ on the deformation to the
normal cone of diagonal, since for $z\in ZU_\hbar,\ m\in M$, the difference of
the left and right actions $z^{(1)}m-z^{(2)}m$ is divisible by $\hbar$.
So we will consider $\varkappa_\hbar(M)$ as a $\BC^\times$-equivariant sheaf on
$\CN_\Delta(\ft/W\times\ft/W)$. Note that $\BC[\CN_\Delta(\ft/W\times\ft/W)]/\hbar=
\BC[C_\Delta(\ft/W\times\ft/W)]=\BC[T_{\ft/W}]=\BC[T_\Sigma]$. Here $\Sigma\subset(\gvee)^*$
is the Kostant slice (we identify $\gvee$ and $(\gvee)^*$ by the Killing form).
Recall the universal
centralizer $\fZ_{\gvee}^{\gvee}=\{(x\in\gvee,\ \xi\in\Sigma) : \on{ad}_x\xi=0\}$.

\begin{lem}
  Under the identification of $\gvee$ and $(\gvee)^*$, the universal centralizer
  $\fZ_{\gvee}^{\gvee}$ is canonically isomorphic to the cotangent bundle $T^*\Sigma$.
\end{lem}

\begin{proof}
  The open subset of regular elements $(\gvee)^*_{\on{reg}}\subset(\gvee)^*$ carries the
  centralizer sheaf $\fz\subset\gvee\otimes\CO$ of abelian Lie algebras. We have
  $\fz=\on{pr}^*T^*\Sigma$ where
  $\on{pr}\colon (\gvee)^*_{\on{reg}}\to(\gvee)^*_{\on{reg}}/\on{Ad}_{G^\vee}=\ft/W=\Sigma$
  is the evident projection. Indeed, the fiber of $\on{pr}^*T^*\Sigma$ at a point
  $\eta\in(\gvee)^*$ is dual to the cokernel of the map $a_\eta\colon \gvee\to(\gvee)^*,\
  x\mapsto\on{ad}_x\eta$. In other words, this fiber is isomorphic to the kernel of the
  dual map which happens to coincide with $a_\eta$. The latter kernel is by definition
  nothing but the fiber $\fz_\eta$.
\end{proof}

\begin{lem}
  For $V\in\on{Rep}(G^\vee)$, the $\BC[T_\Sigma]$-module
  $\varkappa_\hbar(\on{Fr}(V))|_{\hbar=0}$ is isomorphic to $\BC[\Sigma]\otimes V$ as a
  $\fZ_{\gvee}^{\gvee}$-module.
\end{lem}

\begin{proof}
  Let $\on{Poly}((\gvee)^*,\gvee)^{G^\vee}$ be the space of $G^\vee$-invariant polynomial
  morphisms. It is a vector bundle over $\on{Spec}\BC[(\gvee)^*]^{G^\vee}=\Sigma$.
  If $P\in\BC[(\gvee)^*]^{G^\vee}$, then the differential $dP$ is a section of this bundle.
  If $z\in ZU(\gvee)=\BC[(\gvee)^*]^{G^\vee}$, we denote $dz$ by $\sigma_z$.
  If $\{z_i\}$ is a set of generators of $ZU(\gvee)$, then $\{\sigma_{z_i}\}$ forms a
  basis of $\fZ_{\gvee}^{\gvee}$ as a vector bundle over $\Sigma$, and hence identifies
  it with $T^*\Sigma$. Let $z^{(1)},z^{(2)}$ stand for the left and right actions of $z$
  on $\on{Fr}(V)$. We have to check that the action of
  $\frac{z^{(1)}-z^{(2)}}{\hbar}|_{\hbar=0}$ on
  $(U_\hbar\otimes V)|_{\hbar=0}=\BC[(\gvee)^*]\otimes V$ coincides with the action of
  $\sigma_z\in\BC[(\gvee)^*]\otimes\gvee$. But if
  $v\in V,\ z=\sum_{\ul{i}}c_{\ul{i}}x_{i_1}\cdots x_{i_r}\ (x_{i_p}\in\gvee)$, and
  $\tilde{y}\in U_\hbar$ is a lift of $y\in\BC[(\gvee)^*]=U_\hbar|_{\hbar=0}$,
  then $\frac{z(\tilde{y}\otimes v)-(\tilde{y}\otimes v)z}{\hbar}|_{\hbar=0}=
  \sum_{i_p\in\ul{i}}c_{\ul{i}}x_{i_1}\cdots\hat{x}_{i_p}\cdots x_{i_r}y\otimes x_{i_p}(v)=
  \sigma_z(y\otimes v)$.
\end{proof}

\subsection{Quasiclassical limit of the Kostant-Whittaker reduction}

We define a functor $\varkappa\colon \Coh^{G^\vee\times\BC^\times}(\gvee)^*\to
\Coh^{\BC^\times}(T\Sigma)$ as follows. For a $G^\vee$-equivariant coherent sheaf $\CF$
on $(\gvee)^*$, the restriction $\CF|_{(\gvee)^*_{\on{reg}}}$ is equipped with an action
of the centralizer sheaf $\fz$. Hence, this restriction gives rise to a coherent
sheaf on $\on{pr}^*T\Sigma$. Restricting it to the Kostant slice $\Sigma$, we
obtain a coherent sheaf $\bar\varkappa\CF$ on $T\Sigma$. Equivalently, the latter
sheaf can be described as $(\CF|_\Upsilon)^{N^\vee_-}$, where
$\Upsilon=e+\bvee_-\subset\gvee\simeq(\gvee)^*$. This construction is
$\BC^\times$-equivariant, and gives rise to the desired functor $\varkappa$.

We define $\Coh^{G^\vee\times\BC^\times}_{\on{fr}}(\gvee)^*\subset
\Coh^{G^\vee\times\BC^\times}(\gvee)^*$
as the full subcategory formed by the sheaves $V\otimes\CO_{(\gvee)^*}$ for
$V\in\on{Rep}(G^\vee)$.

\begin{lem}
  \textup{(a)} The functors $\varkappa,\varkappa_\hbar$ are exact;

  \textup{(b)} $\varkappa|_{\Coh^{G^\vee\times\BC^\times}_{\on{fr}}(\gvee)^*},\
  \varkappa_\hbar|_{\HC_\hbar^{\on{fr}}}$ are fully faithful.
\end{lem}

\begin{proof}
  The statements about $\varkappa_\hbar$ follow from the ones for $\varkappa$
  by the graded Nakayama Lemma. To prove (a), note that the functor
  $\CF\mapsto\CF|_\Upsilon,\ \Coh^{G^\vee}(\gvee)^*\to\Coh^{N^\vee_-}(\Upsilon)$ is exact.
  Moreover, $N^\vee_-$ acts
  freely on $\Upsilon$, and $\Upsilon/N^\vee_-=\Sigma$. It follows that the functor
  $\CG\mapsto\CG^{N^\vee_-}$ is exact on $\Coh^{N^\vee_-}(\Upsilon)$. Now (b) follows
  since the codimension in $(\gvee)^*$ of the complement
  $(\gvee)^*\setminus(\gvee)^*_{\on{reg}}$ is at least 2, and the centralizer of
  a general regular element is connected.
\end{proof}

\subsection{Equivariant cohomology of the affine Grassmannian}
Now we turn to the topological machinery. We have an evident homomorphism
$\on{pr}^*\colon\BC[\hbar]=H^\bullet_{\BC^\times}(\on{pt})\to
H^\bullet_{G_\CO\rtimes\BC^\times}(\Gr_G)$. Also, viewing
$H^\bullet_{G_\CO\rtimes\BC^\times}(\Gr_G)$ as
$H^\bullet_{\BC^\times}(G_\CO\backslash G_\CK/G_\CO)$, we obtain two homomorphisms
$\on{pr}_1^*,\on{pr}_2^*\colon \BC[\Sigma]=\BC[\ft/W]=
H^\bullet_{G_\CO}(\on{pt})\rightrightarrows H^\bullet_{G_\CO\rtimes\BC^\times}(\Gr_G)$.
Let us assume for simplicity that $G$ is simply connected. Recall the deformation
to the normal cone of diagonal in $\ft/W\times\ft/W$, see~Section~\ref{deformation}.

\begin{prop}
  \label{cohomol}
  There is a natural isomorphism $\alpha\colon\BC[\CN_\Delta(\ft/W\times\ft/W)]\iso
  H^\bullet_{G_\CO\rtimes\BC^\times}(\Gr_G)$ compatible with the above
  $\on{pr}_1^*,\on{pr}_2^*,\on{pr}^*$.
\end{prop}

\begin{proof}
  Since $H^\bullet_{G_\CO\rtimes\BC^\times}(\Gr_G)|_{\hbar=0}=H^\bullet_{G_\CO}(\Gr_G)$, we
  see that $\on{pr}_1^*|_{\hbar=0}=\on{pr}_2^*|_{\hbar=0}$. It follows that
  $(\on{pr}_1^*,\on{pr}_2^*,\on{pr}^*)\colon\BC[\ft/W\times\ft/W\times\BA^1]\to
  H^\bullet_{G_\CO\rtimes\BC^\times}(\Gr_G)$ factors as a composition
  $\BC[\ft/W\times\ft/W\times\BA^1]\to\BC[\CN_\Delta(\ft/W\times\ft/W)]
  \xrightarrow{\alpha}H^\bullet_{G_\CO\rtimes\BC^\times}(\Gr_G)$
  for a uniquely defined homomorphism $\alpha$. Let us check that $\alpha$ is
  injective. Indeed, $\alpha_{\on{loc}}\colon\BC[\CN_\Delta(\ft/W\times\ft/W)]
  \otimes_{\BC[\ft/W\times\BA^1]}\on{Frac}(\BC[\ft\times\BA^1])\to
  H^\bullet_{G_\CO\rtimes\BC^\times}(\Gr_G)\otimes_{\BC[\ft/W\times\BA^1]}
  \on{Frac}(\BC[\ft\times\BA^1])=H^\bullet_{T\times\BC^\times}(\Gr_G)
  \otimes_{\BC[\ft\times\BA^1]}\on{Frac}(\BC[\ft\times\BA^1])\hookrightarrow
  \lim\limits_{\gets}H^\bullet_{T\times\BC^\times}(\ol\Gr{}^\lambda_G)
  \otimes_{\BC[\ft\times\BA^1]}\on{Frac}(\BC[\ft\times\BA^1])=
  \prod_{\lambda\in X_*(T)}\on{Frac}(\BC[\ft\times\BA^1])$ associates to a function
  its restriction to the union of graphs
  $\Gamma_\lambda:=\{(x_1,x_2,c) : x_1=x_2+c\lambda\}\subset\ft\times\ft\times\BA^1$.
  More precisely, for a $T$-fixed point $\lambda\in\Gr_G$, the
  $\BC[\ft\times(\ft/W)\times\BA^1]$-module $H^\bullet_{T\times\BC^\times}(\lambda)$
  is canonically isomorphic to $(\Id,\pi,\Id)_*\CO_{\Gamma_\lambda}$, where $\pi$ stands
  for the projection $\ft\to\ft/W$. Indeed, let $p\colon\Fl_G\to\Gr_G$ be the
  projection from the affine flag variety $\Fl_G=G_\CK/\on{Iw}$ of $G$,
  and let $\tilde\lambda\in\Fl_G$
  be the $T$-fixed point in $p^{-1}(\lambda)$ corresponding to the coweight
  $\lambda\in X_*(T)\subset W_{\on{aff}}$. Viewing $H^\bullet_{T\times\BC^\times}(\Fl_G)$
  as $H^\bullet_{\BC^\times}(\on{Iw}\backslash G_\CK/\on{Iw})$, we identify
  $H^\bullet_{T\times\BC^\times}(\tilde\lambda)$ with a
  $\BC[\ft\times\ft\times\BA^1]$-module $M$ such that
  $(\Id,\pi,\Id)_*M=H^\bullet_{T\times\BC^\times}(\lambda)$. The preimage $T_\lambda$
  of $\tilde\lambda$ in $G_\CK$ is homotopy equivalent to the torus $T$, and
  the action of $T\times T\times\BC^\times$ on $T_\lambda$ is homotopy equivalent
  to $(t_1,t_2,c)\cdot t=t_1tt_2^{-1}\lambda(c)$. We conclude that
  $H^\bullet_{T\times\BC^\times}(\tilde\lambda)=H^\bullet_{\BC^\times}(T\backslash T_\lambda/T)
  =\BC[\Gamma_\lambda]$.

  Finally, the union $\bigcup_{\lambda\in X_*(T)}\Gamma_\lambda$ is Zariski dense in
  $\ft\times\ft\times\BA^1$. Hence $\alpha_{\on{loc}}$ is injective, and $\alpha$
  is injective as well.

  To finish the proof it suffices to compare the graded dimensions of the LHS and
  the RHS (the grading in the LHS arises from the natural action of $\BC^\times$
  on $\ft$ and $\BA^1$). Now $\dim_{\on{gr}}(H^\bullet_{G_\CO\rtimes\BC^\times}(\Gr_G))=
  \dim_{\on{gr}}(H^\bullet_{G_\CO}(\on{pt})\otimes H^\bullet_{\BC^\times}(\on{pt})\otimes
  H^\bullet(\Gr_G))=\dim_{\on{gr}}\BC[x_1,\ldots,x_r,y_1,\ldots,y_r,\hbar]$ where
  $\deg\hbar=2,\ \deg x_i=\deg y_i=2(m_i+1)$, and $m_1,\ldots,m_r$ are the exponents
  of $G$ (so that $m_i+1$ are the degrees of generators of $W$-invariant polynomials
  on $\ft$).

  Furthermore, $\ft/W=\Sigma$, and
  $\CN_\Delta(\Sigma\times\Sigma)\simeq\Sigma\times\Sigma\times\BA^1$ (an affine space).
  Indeed, for affine spaces $V,V'$ we have an isomorphism
  $\beta\colon V\times V'\times\BA^1\iso\CN_V(V\times V')$; namely,
  $\gamma\colon V\times V'\times\BA^1\to V\times V'\times\BA^1,\ (v,v',c)\mapsto
  (v,cv',c)$ factors through $V\times V'\times\BA^1\xrightarrow{\beta}
  \CN_{V\times V'}V\to V\times V'\times\BA^1$. We conclude that
  $\dim_{\on{gr}}(\on{LHS})=\dim_{\on{gr}}(\on{RHS})$. This completes the proof.
  \end{proof}

In view of~Proposition~\ref{cohomol}, we can view
$H^\bullet_{G_\CO\rtimes\BC^\times}(\Gr_G,-)$ as a functor from the full subcategory
$\mathcal{IC}_{G_\CO\rtimes\BC^\times}\subset D^b_{G_\CO\rtimes\BC^\times}(\Gr_G)$
formed by the semisimple complexes,
to $\Coh^{\BC^\times}(\CN_\Delta(\ft/W\times\ft/W))$. This functor is fully faithful
according to~\cite{gins}. For $V\in\on{Rep}(G^\vee)$, one can identify
$H^\bullet_{G_\CO\rtimes\BC^\times}(\Gr_G,S(V))$ with $\varkappa_\hbar(\on{Fr}(V))$;
moreover, one can make this identification compatible with the tensor structures
on $\on{Rep}(G^\vee)$ and $\on{Perv}_{G_\CO}(\Gr_G)$~\cite{befi}:

\begin{thm}
  \label{free Satake}
  The geometric Satake equivalence
  $\BS\colon \on{Rep}(G^\vee)\iso \on{Perv}_{G_\CO}(\Gr_G)$ extends to a tensor
  equivalence $\BS_\hbar\colon \HC_\hbar^{\on{fr}}\iso\mathcal{IC}_{G_\CO\rtimes\BC^\times}$
  such that $\varkappa_\hbar=H^\bullet_{G_\CO\rtimes\BC^\times}(\Gr_G,-)\circ\BS_\hbar$.
  There is also a quasiclassical version
  $\BS_{qc}\colon \Coh_{\on{fr}}^{G^\vee\times\BC^\times}(\gvee)^*\iso\mathcal{IC}_{G_\CO}$
  such that $\varkappa=H^\bullet_{G_\CO}(\Gr_G,-)\circ\BS_{qc}$.
\end{thm}

Now using the formality of RHom-algebras in our categories, one can deduce the
desired derived geometric Satake equivalence. To formulate it, we introduce a bit
of notation. To a dg-algebra $A$ one can associate the triangulated category of
dg-modules localized by quasi-isomorphisms, and a full triangulated subcategory
$D_{\on{perf}}(A)\subset D(A)$ of perfect complexes. Given an algebraic group $H$
acting on $A$, we can consider $H$-equivariant dg-modules and localize them
by quasi-isomorphisms, arriving at the equivariant version $D^H_{\on{perf}}(A)$.

We now consider the dg-versions $\Sym^{[]}(\gvee),U^{[]}_\hbar$ of the graded
algebras $\Sym(\gvee),U_\hbar$, equipping them with the zero differential and
the cohomological grading so that elements of $\gvee$ and $\hbar$ have degree~2.
The construction of the previous paragraph gives rise to the categories
$D^{G^\vee}_{\on{perf}}(U^{[]}_\hbar),D^{G^\vee}_{\on{perf}}(\Sym^{[]}(\gvee))$. Their
Ind-completions will be denoted by
$D^{G^\vee}(U^{[]}_\hbar),D^{G^\vee}(\Sym^{[]}(\gvee))$.
The Ind-completions of the equivariant derived categories
$D^b_{G_\CO\rtimes\BC^\times}(\Gr_G),D^b_{G_\CO}(\Gr_G)$ will be denoted by
$D_{G_\CO\rtimes\BC^\times}(\Gr_G),D_{G_\CO}(\Gr_G)$.

The following theorem is proved in~\cite{befi}.

\begin{thm}
  \label{nonfree}
  The equivalences of~Theorem~\ref{free Satake} extend to the equivalences of
  monoidal triangulated categories $\Psi_\hbar\colon
  D^{G^\vee}_{\on{perf}}(U^{[]}_\hbar)\iso
  D^b_{G_\CO\rtimes\BC^\times}(\Gr_G)$ and $\Psi_{qc}\colon
  D^{G^\vee}_{\on{perf}}(\Sym^{[]}(\gvee))\iso
  D^b_{G_\CO}(\Gr_G)$. They induce the equivalences of their Ind-completions
  $\Psi_\hbar\colon D^{G^\vee}(U^{[]}_\hbar)\iso D_{G_\CO\rtimes\BC^\times}(\Gr_G)$ and
  $\Psi_{qc}\colon D^{G^\vee}(\Sym^{[]}(\gvee))\iso D_{G_\CO}(\Gr_G)$.
\end{thm}

\subsubsection{The dualities}
\label{dualities}
We denote by $\fC_{G^\vee}$ the autoequivalence of $D^{G^\vee}(U^{[]}_\hbar)$
induced by the canonical outer automorphism of $G^\vee$ interchanging conjugacy
classes of $g$ and $g^{-1}$ (the Chevalley involution). We also denote by
${\mathcal C}_G$ the autoequivalence of $D_{G_\CO\rtimes\BC^\times}(\Gr_G)$ induced
by $g\mapsto g^{-1},\ G((z))\to G((z))$. Then the Verdier duality
$\BD\colon D_{G_\CO\rtimes\BC^\times}(\Gr_G)\to D_{G_\CO\rtimes\BC^\times}(\Gr_G)$ and
the duality $D\colon D^{G^\vee}(U^{[]}_\hbar)\to D^{G^\vee}(U^{[]}_\hbar)$ introduced
in~Example~\ref{Toda} are related by
$\Psi_\hbar\circ\fC_{G^\vee}\circ D=\BD\circ\Psi_\hbar$.

\subsection{The regular sheaf}
\label{regular sheaf}
Recall the setup of~Example~\ref{Toda}. We consider $\CalD_\hbar^{[]}(G^\vee)=
U^{[]}_\hbar\ltimes\BC[G^\vee]\in D^{G^\vee}(U^{[]}_\hbar)$. Its image under the
equivalence of~Theorem~\ref{nonfree} is the {\em regular sheaf}
$\CA_R^{\BC^\times}\in D_{G_\CO\rtimes\BC^\times}(\Gr_G)$ isomorphic to
$\bigoplus_{\lambda\in\Lambda^+}\on{IC}(\ol\Gr{}^\lambda_G)\otimes (V^\lambda)^*$.
The quasiclassical analogs are
$D^{G^\vee}(\Sym^{[]}(\gvee))\ni\BC[T^*G^\vee]^{[]}=\Sym^{[]}(\gvee)\otimes\BC[G^\vee]
\mapsto\CA_R\in D_{G_\CO}(\Gr_G)$. One can check that the image of
$\varkappa_\hbar\CalD_\hbar^{[]}(G^\vee)\in D^{G^\vee}(U^{[]}_\hbar)$ under the
equivalence of~Theorem~\ref{nonfree} is the dualizing sheaf
$\bfomega_{\Gr_G}\in D_{G_\CO\rtimes\BC^\times}(\Gr_G)$. It follows that the convolution
algebra of equivariant Borel-Moore homology
$H^{G_\CO\rtimes\BC^\times}_\bullet(\Gr_G)=H^\bullet_{G_\CO\rtimes\BC^\times}(\Gr_G,\bfomega_{\Gr_G})$
is isomorphic to quantum open Toda lattice
$\CT_\hbar=\varkappa_\hbar^{\on{right}}\varkappa_\hbar^{\on{left}}\CalD_\hbar(G^\vee)$
of~Example~\ref{Toda}.

Note that the regular sheaf $\CA_R^{\BC^\times}$ is equipped with an action of
$G^\vee$. Furthermore, the dg-algebra
$\on{RHom}_{D_{G_\CO\rtimes\BC^\times}(\Gr_G)}(\CA_R^{\BC^\times},\CA_R^{\BC^\times})$
is formal, and we have a $G^\vee$-equivariant morphism of dg-algebras
$U^{[]}_\hbar\to\on{RHom}_{D_{G_\CO\rtimes\BC^\times}(\Gr_G)}
(\CA_R^{\BC^\times},\CA_R^{\BC^\times})$
(corresponding to the {\em right} action of $U^{[]}_\hbar$ on
$\CalD_\hbar^{[]}(G^\vee)$). Similarly, the dg-algebra
$\on{RHom}_{D_{G_\CO}(\Gr_G)}(\CA_R,\CA_R)$ is formal,
and we have a $G^\vee$-equivariant morphism of dg-algebras
$\Sym^{[]}(\gvee)\to\on{RHom}_{D_{G_\CO}(\Gr_G)}(\CA_R,\CA_R)$
(corresponding to the {\em right} action of $\Sym^{[]}(\gvee)$ on
$\BC[T^*G^\vee]^{[]}$). Hence for any $\CF\in D_{G_\CO\rtimes\BC^\times}(\Gr_G)$,
the complex $\on{RHom}_{D_{G_\CO\rtimes\BC^\times}(\Gr_G)}(\CA_R^{\BC^\times},\CF)$
carries a structure of $G^\vee$-equivariant dg-module over $U^{[]}_\hbar$.

Thus the functors $\on{RHom}_{D_{G_\CO\rtimes\BC^\times}(\Gr_G)}(\CA_R^{\BC^\times},\bullet),\
\on{RHom}_{D_{G_\CO}(\Gr_G)}(\CA_R,\bullet)$ may be viewed as landing respectively
into $D^{G^\vee}(U^{[]}_\hbar),\ D^{G^\vee}(\Sym^{[]}(\gvee))$.
We will also need their versions
\begin{multline*}
  \Phi_\hbar:=\on{RHom}_{D_{G_\CO\rtimes\BC^\times}(\Gr_G)}({\mathbf 1}_{\Gr_G},
{\mathcal C}_G\CA_R^{\BC^\times}\star\bullet)\iso
\on{RHom}_{D_{G_\CO\rtimes\BC^\times}(\Gr_G)}(\BD\CA_R^{\BC^\times},\bullet)\\
\iso\on{RHom}_{D_{G_\CO\rtimes\BC^\times}(\Gr_G)}(\BC_{\Gr_G},
\CA_R^{\BC^\times}\otimes^!\bullet),
\end{multline*}
and
\begin{multline*}
  \Phi_{qc}:=\on{RHom}_{D_{G_\CO}(\Gr_G)}({\mathbf 1}_{\Gr_G},
{\mathcal C}_G\CA_R\star\bullet)\iso
\on{RHom}_{D_{G_\CO}(\Gr_G)}(\BD\CA_R,\bullet)\\
\iso\on{RHom}_{D_{G_\CO}(\Gr_G)}(\BC_{\Gr_G},\CA_R\otimes^!\bullet).
\end{multline*}
The following lemma
is proved in~\cite{bfn4}.

\begin{lem}
  \label{3.14}
  \textup{(a)}
 The functors $\on{RHom}_{D_{G_\CO\rtimes\BC^\times}(\Gr_G)}(\CA_R^{\BC^\times},\bullet)\colon
  D_{G_\CO\rtimes\BC^\times}(\Gr_G)\to D^{G^\vee}(U^{[]}_\hbar)$ and
  $\on{RHom}_{D_{G_\CO}(\Gr_G)}(\CA_R,\bullet)\colon
  D_{G_\CO}(\Gr_G)\to D^{G^\vee}(\Sym^{[]}(\gvee))$ are canonically isomorphic
  to $\Psi_\hbar^{-1}$ and $\Psi_{qc}^{-1}$ respectively.

  \textup{(b)} The functors $\Phi_\hbar\colon
  D_{G_\CO\rtimes\BC^\times}(\Gr_G)\to D^{G^\vee}(U^{[]}_\hbar)$ and
  $\Phi_{qc}\colon
  D_{G_\CO}(\Gr_G)\to D^{G^\vee}(\Sym^{[]}(\gvee))$ are canonically isomorphic
  to $\fC_{G^\vee}\circ\Psi_\hbar^{-1}$ and $\fC_{G^\vee}\circ\Psi_{qc}^{-1}$ respectively.
\end{lem}

\section{Motivation II: What do we (as mathematicians) might want from 3d N=4 SUSY QFT?
  (naive approach)}\label{naive}

\subsection{Some generalities}
In this Section we would like to introduce some language related to 3-dimensional N=4 super-symmetric quantum field theories.
The reader should be warned from the very beginning: here we are going to use all the words from physics as a ``black box". More precisely, we are not going to try to explain what such a theory really is from a mathematical point of view. Instead we are going to review the relevant ``input data" (i.e.\ to what mathematical structures physicists usually attach such a QFT) and some of the ``output data" (i.e.\ what mathematical structures one should get in the end. This will be largely extended in Section \ref{categorical}, where we partly address the question ``what kind of structures these 3d N=4 SUSY QFTs really are from a mathematical point of view?"). Also it will be important for us to recall (in this Section) what one can do with these theories: i.e. we are going to discuss some operations which produce new quantum field theories out of old ones.

The reader should be warned from the very beginning about the following: both in this Section and in Section \ref{categorical} we are only going to discuss algebraic aspects of the above quantum field theories (such as e.g.\ algebraic varieties or categories one can attach to them). In principle ``true physical theory" is supposed to have some interesting analytic aspects (such as e.g.\ a metric on the above varieties). These analytic aspects will be completely ignored in these notes. Essentially, this means that we are going to study quantum field theories up to certain ``algebraic equivalence" but we are not going to discuss details in these notes.

\subsection{Higgs and Coulomb branch and 3d mirror symmetry}A 3d N=4 super-symmetric quantum field theory $\calT$ is supposed to have a well-defined moduli space of vacua. This should be some complicated (though interesting) space. This space is somewhat too complicated for our present discussion. Instead we are going to discuss some ``easy" parts of this space. Namely, the above moduli space of vacua should have two special pieces called the Higgs and the Coulomb branch of the moduli space of vacua; we shall denote these by $\calM_H(\calT)$ and $\calM_C(\calT)$. They are supposed to be Poisson (generically symplectic) complex algebraic varieties.\footnote{In fact, this is already a simplification: non-algebraic holomorphic symplectic manifolds should also arise in this way, but we are not going to discuss such theories.}
They should also be hyper-k\"ahler in some sense, but (to the best of our knowledge) this notion is not well-defined for singular varieties, we are going to ignore the hyper-k\"ahler structure in these notes. So, in this section we are going to think about a theory $\calT$ in terms of $\calM_H(\calT)$ and $\calM_C(\calT)$. Of course, this is a very small part of what the actual ``physical theory" is, but we shall see that even listing the structures that physicists expect from $\calM_H(\calT)$ and $\calM_C(\calT)$ will lead us to interesting constructions.

One of the operations on theories that will be important in the future is the operation of ``3-dimensional mirror symmetry". Namely, physicists expect that for a theory $\calT$ there should exist a mirror dual theory $\calT^*$ such that
$\calM_H(\calT^*)=\calM_C(\calT)$ and $\calM_C(\calT^*)=\calM_H(\calT)$.

\subsection{More operations on theories}
In what follows we shall use the following notation: for a symplectic variety $X$ with a Hamiltonian $G$-action
we shall denote by $X\tslash G$ the Hamiltonian reduction of $X$ with respect to $G$.

Then the following operations on theories are expected to make sense (in the next subsection we shall start considering examples) :

\medskip
\noindent
1) If $\calT_1,\cdots,\calT_n$ are some theories then one can form their product $\calT_1\times\cdots\times \calT_n$. We have
$$
\calM_H(\calT_1\times\cdots\times \calT_n)=\calM_H(\calT_1)\times\cdots\times \calM_H(\calT_n),
$$
and
$$
\calM_C(\calT_1\times\cdots\times \calT_n)=\calM_C(\calT_1)\times\cdots\times \calM_C(\calT_n).
$$

\medskip
\noindent
2) Let $\calT$ be a theory and let $G$ be a complex reductive group. Then there is a notion of $G$ acting on $\calT$. Physicists say in this case that $G$ maps to the flavor symmetry group of $\calT$, or that we are given a theory $\calT$ with flavor symmetry $G$.

Assume that we are given a theory $\calT$ with flavor symmetry $G$. Then there is a new theory $\calT/G$ obtained by ``gauging" $G$. The origin of the notation is explained in 3).

\medskip
\noindent
3) We have $\calM_H(\calT/G)=\calM_H(\calT)\tslash G$. Of course, the notion of Hamiltonian reduction can be understood in several different ways, so we need to talk a little about what we mean by $\tslash G$ here. Recall that the Hamiltonian reduction is defined as follows. Let $\calX$ be any Poisson variety endowed with a Hamiltonian action of $G$.
Then we have the moment map $\mu: \calX\to \grg^*$. Then we are supposed to have $\calX\tslash G=(\mu^{-1}(0))/G$. Here there are two delicate points. First, the map $\mu$ might not be flat, so honestly we must take $\mu^{-1}(0)$ in the dg-sense. Second, we must specify what we mean by quotient by $G$. In these notes we shall mostly deal with examples when $\calX$ is affine and we shall be primarily interested in the algebra of functions on $\calX\tslash G$. For these purposes it is enough to work with the so called ``categorical quotient", i.e.\ we set
$$
\CC[\calX\tslash G]=(\CC[\mu^{-1}(0)])^G.
$$
Note that according to our conventions this might be a dg-algebra.

\medskip
\noindent
4) Given $\calT$ and $G$ as above (assuming that $G$ is connected and reductive) one can construct a ring object $\calA_\calT$ in $D_{G_{\calO}}(\Gr_G)$ (sometimes we shall denote it by $\calA_{\calT,G}$ when we need to stress the dependence on $G$).
The $!$-stalk of $\calA_\calT$ at the unit point of $\Gr_G$ is $\CC[\calM_C(\calT)]$ and $H^*_{G_{\calO}}(\calA_\calT)=\CC[\calM_C(\calT/G)]$.
In fact, the object $\calA_\calT$ should also have a Poisson structure (which will induce a Poisson structure on $\CC[\calM_C(\calT)]$ and on
 $\calM_C(\calT/G)$) but we are going to ignore this issue for now.

Let us as before denote by $i$ the emebedding of the the point 1 in $\Gr_G$. Then $i^!\calA_\calT$ can be regarded as a ring object
of the equivariant derived category $D_G(\pt)$. Its equivariant cohomology $H^*_G(\pt,i^!\calA_\calT)$ is a graded algebra over $H^*_G(\pt,\CC)=\CC[\grg]^G=\CC[\grt]^W$ whose (derived) fiber over $0$ is equal to $\CC[\calM_C]$. Thus flavor symmetry $G$ is supposed to define a (Poisson) deformation of $\calM_C$ over the base $\grt/W$. In particular, by base change we should have a Poisson deformation of $\calM_C$ over $\grt$.

5) Let $H$ be a subgroup of $G$. Then $\calA_{\calT,H}$ is equal to the $!$-restriction of $\calA_{\calT,G}$ to $\Gr_H$.

\medskip
\noindent
6) For a reductive group $G$ there is a theory $\calT[G]$ such that

a) $\calT[G]$ has flavor symmetry $G$.

b) $\calM_H(\calT[G])=\calN_{\grg}, \calM_C(\calT[G])=\calN_{\grg^{\vee}}$; here $\grg=\on{Lie}(G)$ and $\calN_{\grg}$ is its nilpotent cone.

c) $\calA_{\calT[G^{\vee}]}=\calA_R$ (our ``regular representation" sheaf on $\Gr_G$).

d) $\calT[G]^*=\calT[G^{\vee}]$.

\medskip
\noindent
7) For a theory $\calT$ with flavor symmetry $G$ there should exist another (S-dual) theory $\calT^{\vee}$ with flavor symmetry $G^{\vee}$ (it is defined via S-duality for 4-dimensional N=4 super-symmetric gauge theory).
Gaiotto and Witten \cite{gw} claim that
\begin{equation}\label{gw}
\calT^{\vee}=((\calT\times \calT[G])/ G)^*
\end{equation}
(here the gauging is taken with respect to diagonal copy of $G$).

In particular, the RHS of (\ref{gw}) has an action of $G^{\vee}$ (which a priori is absolutely non-obvious). Here is an example: take $\calT$ to be the trivial theory with trivial $G$-action (in this case $\calM_H=\calM_C=\pt$, but the structure is still somewhat non-trivial as we remember the group $G$). Then $\calT\times \calT[G]=\calT[G]$. Now $\calT[G]/G$ is the theory whose (naive - i.e.\ not dg) Higgs branch is $\pt$ and whose Coulomb branch is isomorphic to $G^{\vee}\times \grt/W=\Spec(H^*_{G_{\calO}}(\Gr_G,\calA_R))$. Since the mirror duality interchanges $\calM_H$ and $\calM_C$ we see that $\calM_H(\calT^{\vee})$ has an action of $G^{\vee}$.

More generally, it follows that
\begin{equation}\label{higgs}
\CC[\calM_H(\calT^{\vee})]=H^*_{G_{\calO}}(\Gr_G,\calA_\calT\overset{!}\otimes \calA_R)
\end{equation}
(this follows from 5). In particular, it has a natural action of $G^{\vee}$.

Let us now pass to examples.

\subsection{Basic example}\label{basic-example}
This is in some sense the most basic example.
Let $\bfM$ be a connected symplectic algebraic variety over $\CC$. Then to $\bfM$ there should correspond a theory $T(\bfM)$ for which
$\calM_H=\bfM$ and $\calM_C=\pt$. In fact, this is true only if dg-structures are disregarded. However, in these notes we shall mostly care about the case when $\bfM$ is just a symplectic vector space and in this case it should be true as stated (cf.\ Subsection \ref{cat-example} for more details).

\subsection{Gauge theories}
Let $G$ be a reductive group. Then an action of $G$ on $T(\bfM)$ should be the same as a Hamiltonian action of $G$ on $\bfM$.\footnote{By Hamiltonian action we mean a symplectic action with fixed moment map.}
Then we can form the theory $T(\bfM)/ G$.

Assume that $\bfM$ is actually a symplectic vector space and that the action of $G$ on $\bfM$ is linear. Then the theory $T(\bfM)/ G$ is called {\em gauge theory with matter $\bfM$}. In the case when $\bfM=\bfN\oplus\bfN^*=T^*\bfN$ for some repesentation $\bfN$ of $G$, the Coulomb branch of these theories together with the corresponding objects $\calA_\calT$ was rigorously defined and studied in \cite{bfn2,bfn3,bfn4}.
Unfortunately, at this point we don't know how to modify our constructions so that they will depend on $\bfM$ rather than on $\bfN$ (but we can check that different ways of representing $\bfM$ as $T^*\bfN$ (in the cases where it is possible) lead to the same $\calM_C$).
We shall sometimes denote the theory $T(\bfM)/ G$ simply by $T(G,\bfN)$.

Here is an interesting source of pairs $(G,\bfN)$ as above. Let $Q$ be an oriented quiver (a.k.a. finite oriented graph) with set of vertices $I$. Let $V$ and $W$ be two finite-dimensional $I$-graded vector spaces over $\CC$. Set
$$
G=\prod\limits_{i\in I} \GL(V_i), \quad \bfN=(\bigoplus \limits_{i\to j} \Hom(V_i,V_j))\oplus(\bigoplus \Hom(V_i,W_i)).
$$
Theories associated with such pairs $(G,\bfN)$ are called {\em quiver gauge theories}. In the case when $Q$ is a quiver of finite Dynkin type the corresponding Coulomb branches are studied in detail in \cite{bfn3}; we review some of these results in Section \ref{quiver}.

\subsection{Toric gauge theories}\label{toric-gauge}
Let $T\subset (\CC^{\times})^n$ be an algebraic torus. We set $T_F=(\CC^{\times})^n/T$ (this is also an algebraic torus).
Then $T$ acts naturally on $\CC^n$, so we can set $\bfN=\CC^n,\ G=T$ in the notation of the previous subsection.

Note that the torus $T_G^{\vee}$ also naturally embeds to $(\CC^\times)^n$ (by dualizing the quotient map $(\CC^\times)^n\to T_F$). It is then expected that the mirror dual to the theory associated to $(T,\CC^n)$ is equal to the theory associated with $(T_F^{\vee},\CC^n)$. Note that this implies that
the Coulomb branch of the former must be isomorphic to $T^*\CC^n\tslash T_F^{\vee}$ (since this is the Higgs branch of the latter). As was mentioned earlier, in the next Section we are going to give a rigorous definition of Coulomb branches for gauge theories of cotangent type and  the above expectation in the toric case is proven in Example \ref{toric-hyp}.

\subsection{Sicilian theories} Let $\Sig$ be a Riemann surface obtained from a compact Riemann surface $\overline{\Sigma}$ by making $n$ punctures. Let also $G$ be a reductive group. To this data physicists associate a theory $T(\Sig,G)$ (``Sicilian theory") with an action of $G^n$.
The construction is by ``compactifying" certain 6-dimensional theory (attached to $G$) on $\Sig\times S^1$.

One of the key statements from physics is that the theory associated to a sphere with $n$-punctures and the group $G^{\vee}$ is
\begin{equation}\label{formula}
(\underbrace{(\calT[G]\times \cdots \times \calT[G])}_{\text{$n$ times}}/ G)^*.
\end{equation}
Here we are gauging the diagonal action of $G$.
It has an action of $(G^{\vee})^n$ for reasons similar to 7).
There should be in fact a simpler statement (when you start with a theory corresponding to any surface and make an additional puncture), but we are a little confused now about what it is.
In particular, it says that functions on the Higgs branch (for any $G$) of (\ref{formula}) is
$H^*_{G_{\calO}}(\Gr_G,\underbrace{\calA_R\overset{!}\otimes\cdots\overset{!}\otimes \calA_R}_{\text{$n$ times}})$ (which is what we knew before for $G$ of type $A$). More precisely, for $G=GL(r)$ the theory $\calT[G]$ is a quiver theory of type $A_{r+1}$ and then the theory $\underbrace{(\calT[G]\times \cdots \times \calT[G])/ G}_{\text{$n$ times}}$ is the corresponding star-shaped quiver theory. Interested reader can consult~\cite[Section 6]{bfn4} for more details.

\subsection{S-duality vs. derived Satake}
Let $\calT$ be a theory with $G$-symmetry and let $\calA_{\calT,G}$ be the corresponding ring object on $\Gr_G$.
We would like to describe the corresponding data for the S-dual theory $\calT^{\vee}$.
Let $\Psi_G$ denote the derived geometric Satake functor going from $D(\Gr_G)$ to the derived category of $G^{\vee}$-equivariant dg-modules
over $\Sym(\grg^{\vee}[-2])$. Then the cohomology $h^*(\Psi_G(\calA_{\calT,G}))$ with grading disregarded can be viewed as a commutative ring object in the
cateogory of $G^{\vee}$-equivariant modules over $\Sym(\grg^{\vee})$. In other words, $\Spec(h^*(\Psi_G(\calA_{\calT,G})))$ is a $G^{\vee}$-scheme endowed with a compatible morphism to $(\grg^{\vee})^*$.

It follows from the results of the previous Section that
$$
\calM_H(\calT^{\vee})=\Spec(h^*(\Psi_G(\calA_{\calT,G}))).
$$

Another (categorical) relationship between the assignment $\calT\mapsto \calT^{\vee}$ and geometric Langlands duality will be discussed in subsection \ref{glanglands}.

\section{Coulomb branches of 3-dimensional gauge theories}\label{Coulomb}
In this Section we explain how to define Coulomb branches (and some further structures related to flavor symmetry) for gauge theories of cotangent type.

\subsection{Summary}
Let us summarize what is done in this Section. Let $G$ be a complex connected reductive group and let $\bfN$ be a representation of it.
In this Section we are going to define mathematically the Coulomb branch $\calM_C(G,\bfN)$ of the gauge theory $\calT(T^*\bfN)/ G$. These Coulomb branches will satisfy the following (non-exhaustive) list of properties:

(1) $\calM_C(G,\bfN)$ is a normal, affine generically symplectic Poisson variety (conjecturally it is singular symplectic but we don't know how to prove this).

(2) Let $T$ be a maximal torus in $G$ and let $W$ be the Weyl group of $G$. Then $\calM_C(G,\bfN)$ is birationally isomorphic to $(T^*T^{\vee})/W$. This birational isomorphism is compatible with the Poisson structure. In particular, $\dim (\calM_C(G,\bfN)=2\on{rank} G$.

(3) There is a natural ``integrable system" map $\pi:\calM_C(G,\bfN)\to \grt/W$ which has Lagrangian fibers.

(4) $\calM_C(G,\bfN)$ is equipped with a canonical quantization; the map $\pi$ also gets quantized.

\subsection{General setup}
\label{general setup}
Let $\bN$ be a finite dimensional representation of a complex connected
reductive group $G$.
We consider the moduli space $\CR_{G,\bN}$ of triples $(\CP,\sigma,s)$
where $\CP$ is a $G$-bundle on the formal disc $D=\on{Spec}\CO;\ \sigma$
is a trivialization of $\CP$ on the punctured formal disc $D^*=\on{Spec}\CK$;
and $s$ is a section of the associated vector bundle
$\CP_{\on{triv}}\Gtimes\bN$ on $D^*$ such that $s$ extends to a
regular section of $\CP_{\on{triv}}\Gtimes\bN$ on $D$, and
$\sigma(s)$ extends to a regular section of $\CP\Gtimes\bN$ on
$D$. In other words, $s$ extends to a regular section of the vector bundle
associated to the $G$-bundle
glued from $\CP$ and $\CP_{\on{triv}}$ on the non-separated formal scheme glued
from $2$ copies of $D$ along $D^*$ ({\em raviolo}).
The group $G_\CO$ acts on $\CR_{G,\bN}$ by changing the trivialization
$\sigma$, and we have an evident
$G_\CO$-equivariant projection $\CR_{G,\bN}\to\Gr_G$ forgetting $s$.
The fibers of this projection are profinite dimensional vector spaces:
the fiber over the base point is $\bN\otimes\CO$, and all the other fibers
are subspaces in $\bN\otimes\CO$ of finite codimension. One may say that
$\CR_{G,\bN}$ is a $G_\CO$-equvariant ``constructible profinite dimensional
vector bundle'' over $\Gr_G$.

\subsubsection{Example: affine Steinberg variety}
\label{steinberg}
If $\bN$ is the adjoint representation $G\curvearrowright\fg$, then $\CR_{G,\bN}$ is isomorphic
to the union
$\bigcup_{\lambda\in\Lambda^+}T^*_{\Gr^\lambda_G}\Gr_G$
of conormal bundles to the $G_\CO$-orbits in $\Gr_G$.

\bigskip

The $G_\CO$-equivariant Borel-Moore homology
$H^{G_\CO}_\bullet(\CR_{G,\bN})$ is defined via the following limiting procedure.

We define $\CR_{\leq\lambda}$ as the preimage of $\ol\Gr{}^\lambda_G$ in
$\CR:=\CR(G,\bN)$.
It suffices to define the $G_\CO$-equivariant Borel-Moore homology
$H_\bullet^{G_\CO}(\CR_{\leq\lambda})$  along with the maps
$H_\bullet^{G_\CO}(\CR_{\leq\lambda})\to H_\bullet^{G_\CO}(\CR_{\leq\mu})$
for $\lambda\leq\mu$. For a fixed $\lambda$ and $d\gg0,\ \CR_{\leq\lambda}$ is
invariant under the translations by  $z^d\bN_\CO$, and we denote the quotient by
$\CR_{\leq\lambda}^d$, so that
$\CR_{\leq\lambda}=\lim\limits_{\gets}\CR_{\leq\lambda}^d$.
For fixed $\lambda,d$, and $e\gg0$, the action of
$G_\CO$ on $\CR_{\leq\lambda}^d$
factors through the action of $G_{\CO/z^e\CO}$. Finally,
\begin{equation*}
H_\bullet^{G_\CO}(\CR_{\leq\lambda}):=H^{-\bullet}_{G_{\CO/z^e\CO}}(\CR_{\leq\lambda}^d,
\bfomega_{\CR_{\leq\lambda}^d})[-2\dim(\bN_\CO/z^d\bN_\CO)].
\end{equation*}
The cohomological shift means that we are considering the ``renormalized''
Borel-Moore homology, i.e.\ the cohomology
$H^{-\bullet}_{G_\CO}(\CR,\bfomega_\CR[-2\dim\bN_\CO])$.

The $G_\CO$-equivariant Borel-Moore homology
$H^{G_\CO}_\bullet(\CR_{G,\bN})$ forms an associative
algebra with respect to the following convolution operation.
We consider the diagram
\begin{equation}\label{3.2}
    \begin{CD}
        \CR \times\CR @<\tilde p<< p^{-1}(\CR\times\CR)
        @>\tilde q>> q(p^{-1}(\CR\times\CR))
        @>\tilde m>> \CR
        \\
        @V{i\times\Id_\CR}VV @V{i'}VV @VVV @VV{i}V
        \\
        \CT\times \CR @<p<< G_\CK\times\CR @>q>>
        G_\CK\stackrel{G_\CO}{\times}\CR @>m>> \CT,
    \end{CD}
\end{equation}
Here $\CT:=G_\CK\stackrel{G_\CO}{\times}\bN_\CO$, and we have an embedding
$\CT\hookrightarrow\Gr_G\times\bN_\CK$ such that
$\CR=\CT\cap(\Gr_G\times\bN_\CO)$. The embedding $\CR\hookrightarrow\CT$
is denoted by $i$. The maps in the lower row are given by
\begin{equation*}
  %\left([g_1,g_2s], [g_2,s]\right) \mapsfrom
  \bigl(g_1,[g_2,s]\bigr)  \stackrel{q}{\mapsto}
  \bigl[g_1, [g_2,s]\bigr] \stackrel{m}{\mapsto} [g_1g_2, s],\
  \bigl(g_1,[g_2,s]\bigr)  \stackrel{p}{\mapsto}
  \bigl([g_1,g_2s], [g_2,s]\bigr),
  \end{equation*}
and all the squares are cartesian (i.e.\ the upper row consists of closed subvarietes
in the lower row, and all the maps in the upper row are induced by the corresponding
maps in the lower row). We have the following group actions on the terms of the
lower row preserving the closed subvarieties in the upper row:
\begin{equation*}
    \begin{split}
        G_\CO \times G_\CO\curvearrowright
        \CT\times\CR ;\ &
        (g,h)\cdot \left([g_1,s_1],[g_2,s_2]\right)
        = \left([gg_1, s_1], [hg_2,s_2] \right),
        \\
        G_\CO \times G_\CO \curvearrowright G_\CK\times\CR ;\ &
        (g,h)\cdot \left(g_1, [g_2,s]\right)
        = \left(gg_1h^{-1}, [hg_2, s]\right),
        \\
       G_\CO \curvearrowright G_\CK\stackrel{G_\CO}{\times}\CR ;\ &
        g\cdot \left[g_1, [g_2,s]\right] = \left[gg_1, [g_2, s]\right],
        \\
        G_\CO \curvearrowright \CT ;\ &
        g\cdot [g_1,s] = [gg_1, s].
    \end{split}
\end{equation*}
The morphisms $p,q,m$ (and hence $\tilde{p},\tilde{q},\tilde{m}$)
are equivariant, where we take the first projection
$\on{pr}_1\colon G_\CO\times G_\CO\to G_\CO$ for $q$.

Finally, given two equivariant Borel-Moore homology classes
$c_1,c_2\in H_\bullet^{G_\CO}(\CR)$, we define their convolution product
$c_1*c_2:=\tilde{m}_*(\tilde{q}{}^*)^{-1}\tilde{p}{}^*(c_1\otimes c_2)$.

This algebra is commutative,
finitely generated and integral, and its spectrum
$\CM_C(G,\bN)=\on{Spec}H^{G_\CO}_\bullet(\CR_{G,\bN})$ is an irreducible
normal affine variety of dimension $2\on{rk}(G)$, the {\em Coulomb branch}.
It is supposed to be a (singular) hyper-K\"ahler manifold~\cite{sw}.

Let $T\subset G$ be a Cartan torus with Lie algebra $\ft\subset\fg$.
Let $W=N_G(T)/T$ be the corresponding Weyl group. Then the equivariant
cohomology $H^\bullet_{G_\CO}(\on{pt})=\BC[\ft/W]$ forms a subalgebra of
$H^{G_\CO}_\bullet(\CR_{G,\bN})$ (a {\em Cartan subalgebra}), so we have a
projection $\varPi\colon \CM_C(G,\bN)\to\ft/W$.

\subsubsection{Example}
For the adjoint representation $\bN=\fg$ considered in~\ref{steinberg}, we
get $\CM_C(G,\fg)=(T^\vee\times\ft)/W$. For the trivial representation, we
get $\CM_C(G,0)=\fZ_{\gvee}^{G^\vee}=\{(g\in G^\vee,\ \xi\in\Sigma) :
\on{Ad}_g\xi=\xi\}$, the universal centralizer of the dual group.

\bigskip

Finally, the algebra $H^{G_\CO}_\bullet(\CR_{G,\bN})$ comes equipped with
quantization: a $\BC[\hbar]$-deformation
$\BC_\hbar[\CM_C(G,\bN)]=H^{G_\CO\rtimes\BC^\times}_\bullet(\CR_{G,\bN})$
where $\BC^\times$ acts by loop rotations, and
$\BC[\hbar]=H^\bullet_{\BC^\times}(\on{pt})$. It gives rise to a Poisson bracket
on $\BC[\CM_C(G,\bN)]$ with an open symplectic leaf, so that $\varPi$
becomes an integrable system: $\BC[\ft/W]\subset\BC[\CM_C(G,\bN)]$ is a
Poisson-commutative polynomial subalgebra with $\on{rk}(G)$ generators.

\subsection{Monopole formula}
\label{monopole}
Recall that $\CR_{G,\bN}$
is a union of (profinite dimensional) vector bundles over $G_\CO$-orbits
in $\Gr_G$. The corresponding Cousin spectral sequence converging to
$H_\bullet^{G_\CO}(\CR_{G,\bN})$ degenerates and allows to compute the
equivariant Poincar\'e polynomial (or rather Hilbert series)
\begin{equation}
\label{mono}
P_t^{G_\CO}(\CR_{G,\bN})=
\sum_{\theta\in\Lambda^+}t^{d_\theta-2\langle\rhovee,\theta\rangle}P_G(t;\theta).
\end{equation}
Here $\deg(t)=2,\ P_G(t;\theta)=\prod(1-t^{d_i})^{-1}$ is the Hilbert series
of the equivariant cohomology $H^\bullet_{\on{Stab}_G(\theta)}(\on{pt})\ (d_i$
are the degrees of generators of the ring of $\on{Stab}_G(\theta)$-invariant
functions on its Lie algebra), and $d_\theta=
\sum_{\chi\in\Lambda^\vee_G}\on{max}(-\langle\chi,\theta\rangle,0)\dim\bN_\chi$.
This is a slight variation of the {\em monopole formula} of~\cite{chz}.
Note that the series~(\ref{mono}) may well diverge (even as a formal Laurent
series: the space of homology of given degree may be infinite-dimensional),
e.g.\ this is always the case for unframed quiver gauge theories.
To ensure its convergence (as a formal Taylor series with the constant term~1)
one has to impose the so called `good' or `ugly' assumption on the theory.
In this case the resulting $\BN$-grading on $H_\bullet^{G_\CO}(\CR_{G,\bN})$
gives rise to a $\BC^\times$-action on $\CM_C(G,\bN)$, making it a conical
variety with a single (attracting) fixed point.

\subsection{Flavor symmetry}
\label{flavor}
Suppose we have an extension $1\to G\to\tilde G\to G_F\to1$ where $G_F$
is a connected reductive group (a {\em flavor group}), and the action of
$G$ on $\bN$ is extended to an action of $\tilde G$. Then the action of
$G_\CO$ on $\CR_{G,\bN}$ extends to an action of $\tilde G_\CO$, and the
convolution product defines a commutative algebra structure on the equivariant
Borel-Moore homology $H^{\tilde G_\CO}_\bullet(\CR_{G,\bN})$. We have the
restriction homomorphism $H^{\tilde G_\CO}_\bullet(\CR_{G,\bN})\to
H^{G_\CO}_\bullet(\CR_{G,\bN})=H^{\tilde G_\CO}_\bullet(\CR_{G,\bN})
\otimes_{H^\bullet_{G_F}(\on{pt})}\BC$. In other words, $\ul\CM{}_C(G,\bN):=
\Spec H^{\tilde G_\CO}_\bullet(\CR_{G,\bN})$ is a deformation
of $\CM_C(G,\bN)$ over $\Spec H^\bullet_{G_F}(\on{pt})=\ft_F/W_F$.

We will need the following version of this construction. Let $\sZ\subset G_F$
be a torus embedded into the flavor group. We denote by $\tilde G^\sZ$
the pullback extension $1\to G\to\tilde G^\sZ\to\sZ\to1$. We define
$\ul\CM{}_C^\sZ(G,\bN):=\Spec H^{\tilde G^\sZ_\CO}_\bullet(\CR_{G,\bN})$:
a deformation of $\CM_C(G,\bN)$ over $\fz:=\Spec H^\bullet_\sZ(\on{pt})$.

Since $\CM_C(G,\bN)$ is supposed to be a hyper-K\"ahler manifold, its
flavor deformation should come together with a (partial) resolution.
To construct it, we consider the obvious projection
$\tilde\pi\colon \CR_{\tilde G,\bN}\to\Gr_{\tilde G}\to\Gr_{G_F}$.
Given a dominant coweight $\lambda_F\in\Lambda^+_F\subset\Gr_{G_F}$, we set
$\CR_{\tilde G,\bN}^{\lambda_F}:=\tilde\pi^{-1}(\lambda_F)$, and consider
the equivariant Borel-Moore homology
$H^{\tilde G{}^\sZ_\CO}_\bullet(\CR_{\tilde G,\bN}^{\lambda_F})$. It carries a convolution
module structure over $H^{\tilde G{}^\sZ_\CO}_\bullet(\CR_{G,\bN})$. We consider
$\wt{\ul\CM}{}_C^{\sZ,\lambda_F}(G,\bN):=
\Proj(\bigoplus_{n\in\BN}H^{\tilde G{}^\sZ_\CO}_\bullet(\CR_{\tilde G,\bN}^{n\lambda_F}))
\stackrel{\varpi}{\longrightarrow}\ul\CM{}_C^\sZ(G,\bN)$.
We denote $\varpi^{-1}(\CM_C(G,\bN))$ by $\wt\CM_C^{\lambda_F}(G,\bN)$.
We have $\wt\CM_C^{\lambda_F}(G,\bN)=
\Proj(\bigoplus_{n\in\BN}H^{G_\CO}_\bullet(\CR_{\tilde G,\bN}^{n\lambda_F}))$.

More generally, for a strictly convex (i.e.\ not containing nontrivial
subgroups) cone $\sV\subset\Lambda^+_F$, we consider the multi projective
spectra $\wt{\ul\CM}{}_C^{\sZ,\sV}(G,\bN):=
\Proj(\bigoplus_{\lambda_F\in\sV}H^{\tilde G{}^\sZ_\CO}_\bullet(\CR_{\tilde G,\bN}^{\lambda_F}))
\stackrel{\varpi}{\longrightarrow}\ul\CM{}_C^\sZ(G,\bN)$ and
$\wt\CM_C^\sV(G,\bN):=
\Proj(\bigoplus_{\lambda_F\in\sV}H^{G_\CO}_\bullet(\CR_{\tilde G,\bN}^{\lambda_F}))
\stackrel{\varpi}{\longrightarrow}\CM_C(G,\bN)$.

The following proposition is proved in~\cite{bfn2}.
\begin{prop}
\label{coulomb reduction}
Assume that the flavor group is a torus, i.e.\ we have an exact sequence
$1\to G\to\tilde G\to T_F\to1$. Then the Coulomb branch $\CM_C(G,\bN)$ is
the Hamiltonian reduction of $\CM_C(\tilde G,\bN)$ by the action of the
dual torus $T_F^\vee$.
\end{prop}

\subsubsection{Example: toric hyper-K\"ahler manifolds}\label{toric-hyp}
Consider an exact sequence
\begin{equation*}
0\to\BZ^{d-n}\xrightarrow{\alpha}\BZ^d\xrightarrow{\beta}\BZ^n\to0
\end{equation*}
and the associated sequence
\begin{equation}
\label{4.11}
1\to G=(\BC^\times)^{d-n}\xrightarrow{\alpha}\tilde{G}=(\BC^\times)^d\xrightarrow{\beta}
T_F=(\BC^\times)^n\to1
\end{equation}
Let $\bN=\BC^d$ considered as a representation of $G$ via $\alpha$.
By~Proposition~\ref{coulomb reduction}, the Coulomb branch
$\CM_C(G,\bN)$ is the Hamiltonian reduction of $\CM_C((\BC^\times)^d,\BC^d)$ by
the action of $T_F^\vee$. It is easy to see that
$\CM_C((\BC^\times)^d,\BC^d)=\CM_C(\BC^\times,\BC)^d\simeq\BA^{2d}$,
and hence $\CM_C(G,\bN)$ is, by definition, the toric hyper-K\"ahler manifold
associated with the dual sequence of~(\ref{4.11})~\cite{bida}.

In particular, if $\bN$ is a 1-dimensional representation of $\BC^\times$
with the character $q^n$, then $\CM_C(\BC^\times,\bN)$ is the Kleinian surface
of type $A_{n-1}$ given by the equation $xy=w^n$. If $\bN$ is an $n$-dimensional
representation of $\BC^\times$ with the character $nq$, then the Coulomb branch
$\CM_C(\BC^\times,\bN)$ is the same Kleinian surface of type $A_{n-1}$.

\subsection{Ring objects in the derived Satake category}
\label{ring objects}
Let $\pi$ stand for the projection $\CR\to\Gr_G$. Then
$\CA^{\BC^\times}:=\pi_*\bfomega_\CR[-2\dim\bN_\CO]$ is an object of
$D_{G_\CO\rtimes\BC^\times}(\Gr_G)$, and
$H^\bullet_{G_\CO\rtimes\BC^\times}(\CR,\bfomega_\CR[-2\dim\bN_\CO])=
H^\bullet_{G_\CO\rtimes\BC^\times}(\Gr_G,\CA)$. One can equip $\CA^{\BC^\times}$ with a structure
of a ring object in $D_{G_\CO\rtimes\BC^\times}(\Gr_G)$ so that the resulting ring
structure on $H^\bullet_{G_\CO\rtimes\BC^\times}(\Gr_G,\CA^{\BC^\times})$ coincides with the ring
structure on $H_\bullet^{G_\CO\rtimes\BC^\times}(\CR)$ introduced
in~Section~\ref{general setup}. If we forget the loop rotation equivariance,
then the resulting ring object $\CA$ of $D_{G_\CO}(\Gr_G)$ is commutative.

Similarly, in the situation of~Section~\ref{flavor}, we denote
$\tilde\CR:=\CR(\tilde{G},\bN)$, and consider the composed projection
$\tilde\pi\colon \tilde\CR\to\Gr_{\tilde G}\to\Gr_{G_F}$. We define a ring object
$\CA_F^{\BC^\times}:=\on{Ind}_{\tilde{G}_\CO\rtimes\BC^\times}^{(G_F)_\CO\rtimes\BC^\times}
\tilde{\pi}_*\bfomega_{\tilde\CR}[-2\dim\bN_\CO]\in
D_{(G_F)_\CO\rtimes\BC^\times}(\Gr_{G_F})$, where
$\on{Ind}_{\tilde{G}_\CO\rtimes\BC^\times}^{(G_F)_\CO\rtimes\BC^\times}$
is the functor changing equivariance from $\tilde{G}_\CO\rtimes\BC^\times$
to $(G_F)_\CO\rtimes\BC^\times$. If we forget the loop rotation equivariance,
we obtain a commutative ring object $\CA_F\in D_{(G_F)_\CO}(\Gr_{G_F})$.
We will also need the fully equivariant ring object
$\tilde\CA{}_F^{\BC^\times}:=\tilde{\pi}_*\bfomega_{\tilde\CR}[-2\dim\bN_\CO]\in
D_{\tilde{G}_\CO\rtimes\BC^\times}(\Gr_{G_F})$.

The ring $\BC_\hbar[\CM_C(G,\bN)]$ is reconstructed from the ring object
$\tilde\CA{}_F^{\BC^\times}$ by the following procedure going back to~\cite{abg}.
For a flavor coweight $\lambda_F$ we denote by $i_{\lambda_F}$ the
embedding of a $T_F$-fixed point $\lambda_F$ into $\Gr_{G_F}$.
Then $\Ext^\bullet_{D_{\tilde{G}_\CO\rtimes\BC^\times}(\Gr_{G_F})}
({\mathbf 1}_{\Gr_{G_F}},\tilde\CA{}_F^{\BC^\times})=i_0^!\tilde\CA{}_F^{\BC^\times}\simeq
H^\bullet_{\tilde{G}_\CO\rtimes\BC^\times}(\CR,\bfomega_\CR[-2\dim\bN_\CO])$
by the base change. Given
$x,y\in\Ext^\bullet_{D_{\tilde{G}_\CO\rtimes\BC^\times}(\Gr_{G_F})}
({\mathbf 1}_{\Gr_{G_F}},\tilde\CA{}_F^{\BC^\times})$, we consider
$x\star y\in\Ext^\bullet_{D_{\tilde{G}_\CO\rtimes\BC^\times}(\Gr_{G_F})}
({\mathbf 1}_{\Gr_{G_F}}\star{\mathbf 1}_{\Gr_{G_F}},
\tilde\CA{}_F^{\BC^\times}\star\tilde\CA{}_F^{\BC^\times})$, and then apply the isomorphism
${\mathbf 1}_{\Gr_{G_F}}\simeq{\mathbf 1}_{\Gr_{G_F}}\star{\mathbf 1}_{\Gr_{G_F}}$
and the multiplication morphism
$\sfm\colon \tilde\CA{}_F^{\BC^\times}\star\tilde\CA{}_F^{\BC^\times}\to\tilde\CA{}_F^{\BC^\times}$
in order to obtain $\sfm(x\star y)\in
\Ext^\bullet_{D_{\tilde{G}_\CO\rtimes\BC^\times}(\Gr_{G_F})}
({\mathbf 1}_{\Gr_{G_F}},\tilde\CA{}_F^{\BC^\times})$. It is proved in~\cite{bfn4} that
the resulting ring structure on
$\Ext^\bullet_{D_{\tilde{G}_\CO\rtimes\BC^\times}(\Gr_{G_F})}
({\mathbf 1}_{\Gr_{G_F}},\tilde\CA{}_F^{\BC^\times})=H_\bullet^{\tilde{G}_\CO\rtimes\BC^\times}(\CR)$
induces the one introduced in~Section~\ref{general setup} on
$H_\bullet^{G_\CO\rtimes\BC^\times}(\CR)$. Moreover, a similar
construction defines a multiplication
$i_{\lambda_F}^!\bar\CA_F^{\BC^\times}\otimes i_{\mu_F}^!\bar\CA_F^{\BC^\times}\to
i_{\lambda_F+\mu_F}^!\bar\CA_F^{\BC^\times}$ for $\lambda_F,\mu_F\in\Lambda^+_F$.
Here $\bar\CA_F^{\BC^\times}=\on{Res}_{\tilde{G}_\CO\rtimes\BC^\times}^{G_\CO\rtimes\BC^\times}
\tilde\CA{}_F^{\BC^\times}$ is obtained from $\tilde\CA{}_F^{\BC^\times}$ applying
the functor restricting equivariance from
$\tilde{G}_\CO\rtimes\BC^\times$ to $G_\CO\rtimes\BC^\times$.
In particular, we get a module structure
$i_0^!\bar\CA_F^{\BC^\times}\otimes i_{\lambda_F}^!\bar\CA_F^{\BC^\times}\to
i_{\lambda_F}^!\bar\CA_F^{\BC^\times}$. Note that
$i_0^!\bar\CA_F^{\BC^\times}\simeq H_\bullet^{G_\CO\rtimes\BC^\times}(\CR)$.

\subsubsection{Example: The regular sheaf in type A}
Let $G=\GL(\BC^{N-1})\times\GL(\BC^{N-2})\times\ldots\times\GL(\BC^1),\
\tilde{G}=(G\times\GL(\BC^N))/Z$, where $Z\simeq\BC^\times$ is the diagonal central
subgroup. Hence $G_F=\on{PGL}(\BC^N)$. Furthermore,
$\bN=\Hom(\BC^N,\BC^{N-1})\oplus\Hom(\BC^{N-1},\BC^{N-2})\oplus\ldots\oplus
\Hom(\BC^2,\BC^1)$. It is proved in~\cite{bfn4} that
$\CA_F^{\BC^\times}$ is isomorphic to the regular sheaf
$\CA_R^{\BC^\times}\in D_{\on{PGL}(\BC^N)_\CO\rtimes\BC^\times}(\Gr_{\on{PGL}(\BC^N)})$
of~Section~\ref{regular sheaf}.

\subsection{Gluing construction}
Let $\CA^{\BC^\times}_1,\ldots,\CA^{\BC^\times}_n$ be the ring objects
in $D_{G_\CO\rtimes\BC^\times}(\Gr_G)$. We denote the ring objects of
$D_{G_\CO}(\Gr_G)$ obtained by forgetting the loop rotation equivariance
by $\CA_1,\ldots,\CA_n$. Let
$i_\Delta\colon \Gr_G\hookrightarrow\prod_{k=1}^n\Gr_G$ be the diagonal
embedding. The following proposition is proved in~\cite{bfn4}.

\begin{prop}
  $\CA^{\BC^\times}:=i^!_\Delta(\boxtimes\CA^{\BC^\times}_k)$ is a ring object in
  $D_{G_\CO\rtimes\BC^\times}(\Gr_G)$. If the ring objects $\CA_1,\ldots,\CA_n$
  are commutative, then $\CA:=i^!_\Delta(\boxtimes\CA_k)\in D_{G_\CO}(\Gr_G)$
  is a commutative ring object. In particular, the ring
  $H^\bullet_{G_\CO}(\Gr_G,\CA)$ is commutative.
\end{prop}

\begin{proof}
We have $\boxtimes\mathsf m\colon
(\boxtimes\CA_k)\star(\boxtimes\CA_k) = \boxtimes (\CA_k \star
\CA_k) \to \boxtimes \CA_k$ from $\mathsf m\colon
\CA_k\star\CA_k\to\CA_k$. Then we apply $i_\Delta^!$.
We claim that there is a natural homomorphism
\begin{equation*}
i_\Delta^!(\boxtimes\CA_k)\star i_\Delta^!(\boxtimes\CA_k)
\to i_\Delta^! \left(\boxtimes (\CA_k\star\CA_k)\right),
\end{equation*}
hence its composition with $i_\Delta^!(\boxtimes\mathsf m)$ gives the
desired multiplication homomorphism of $i_\Delta^!(\boxtimes\CA_k)$.
We prove the claim by comparing the convolution diagrams~(\ref{convolut})
for $\Gr_G$ and $\prod_k\Gr_G$. Since $p$, $q$ are smooth, $p^*$, $q^*$
commute with $i_\Delta^!$.
The last part of the convolution diagram for $G$ and $\prod_k G$ is
\begin{equation*}
  \begin{CD}
    \Gr_G\tilde\times\Gr_G @>m>>\Gr_G \\
    @V{i'_\Delta}VV @VV{i_\Delta}V \\
    \prod_k \Gr_G\tilde\times\Gr_G
    = \Gr_{\prod_k G}\tilde\times\Gr_{\prod_k G}
    @>>\prod_k m> \Gr_{\prod_k G} = \prod_k \Gr_G,
  \end{CD}
\end{equation*}
where we denote the diagonal embedding of the left column by
$i'_\Delta$ to distinguish it from the right column. Let
\(
\boxtimes (\CA_k \tilde\boxtimes \CA_k)
\)
denote the complex on
\(
\Gr_{\prod_k G}\tilde\times\Gr_{\prod_k G}
\)
obtained in the course of the convolution product for $\prod_k G$.
We define the homomorphism as
\begin{equation*}
\begin{aligned}
    m_* i_\Delta^{\prime !}(\boxtimes (\CA_k \tilde\boxtimes \CA_k))
    = m_* \bigotimes^! (\CA_k \tilde\boxtimes \CA_k)
    \to\\
    \bigotimes^! m_* (\CA_k \tilde\boxtimes \CA_k) =
    i_\Delta^{!}(\prod_k m)_* \boxtimes (\CA_k \tilde\boxtimes \CA_k)).
\end{aligned}
\end{equation*}
\end{proof}

Recall the regular sheaf $\CA_R^{\BC^\times}$ of~Section~\ref{regular sheaf}.
It is equipped with an action of $G^\vee\ltimes U_\hbar^{[]}$. Hence for any
ring object $\CA^{\BC^\times}\in D_{G_\CO\rtimes\BC^\times}(\Gr_G)$, the product
$\CA_R^{\BC^\times}\otimes^!\CA^{\BC^\times}$ is also equipped with an action of
$G^\vee\ltimes U_\hbar^{[]}$. The cohomology ring
$H^\bullet_{G_\CO\rtimes\BC^\times}(\Gr_G,\CA_R^{\BC^\times}\otimes^!\CA^{\BC^\times})$ is also
equipped with an action of $G^\vee\ltimes U_\hbar^{[]}$. The following proposition
is proved in~\cite{bfn4} (recall that the autoequivalence $\fC_{G^\vee}$ was
defined in~Section~\ref{dualities}):

\begin{prop}
  \label{fusion}
  For ring objects $\CA^{\BC^\times}_1,\CA^{\BC^\times}_2\in D_{G_\CO\rtimes\BC^\times}(\Gr_G)$,
  we have
  \begin{equation*}
  \begin{aligned}
    &H^\bullet_{G_\CO\rtimes\BC^\times}(\Gr_G,\CA_1^{\BC^\times}\otimes^!\CA_2^{\BC^\times})\simeq\\
    H^\bullet_{G_\CO\rtimes\BC^\times}(\Gr_G,\CA_R^{\BC^\times}&\otimes^!\CA_1^{\BC^\times})\otimes
    \fC_{G^\vee}H^\bullet_{G_\CO\rtimes\BC^\times}(\Gr_G,\CA_R^{\BC^\times}\otimes^!\CA_2^{\BC^\times})
    /\!\!/\!\!/\Delta_{G^\vee}
    \end{aligned}
  \end{equation*}
  (quantum Hamiltonian reduction). If the ring objects
  $\CA_1,\CA_2\in D_{G_\CO}(\Gr_G)$ obtained by forgetting the loop rotation
  equivariance are commutative, then we have a similar isomorphism of commutative
  rings:
    \begin{equation*}
    H^\bullet_{G_\CO}(\Gr_G,\CA_1\otimes^!\CA_2)\simeq
    H^\bullet_{G_\CO}(\Gr_G,\CA_R\otimes^!\CA_1)\otimes
    \fC_{G^\vee}H^\bullet_{G_\CO}(\Gr_G,\CA_R\otimes^!\CA_2)/\!\!/\!\!/\Delta_{G^\vee}
  \end{equation*}
\end{prop}

\begin{proof}
  By rigidity, we have
  \begin{multline*}
    H^\bullet_{G_\CO\rtimes\BC^\times}(\Gr_G,\CA^{\BC^\times}_1\otimes^!\CA^{\BC^\times}_2)=
    \Ext^\bullet_{D_{G_\CO\rtimes\BC^\times}(\Gr_G)}
    (\BD\CA^{\BC^\times}_1,\CA^{\BC^\times}_2)\\
   = \Ext^\bullet_{D_{G_\CO\rtimes\BC^\times}(\Gr_G)}
   ({\mathbf 1}_{\Gr_G},{\mathcal C}_G\CA^{\BC^\times}_1\star\CA^{\BC^\times}_2)=\\
   \Ext^\bullet_{D^{G^\vee}(U_\hbar^{[]})}\left(U_\hbar^{[]},
   \fC_{G^\vee}\Psi_\hbar^{-1}(\CA^{\BC^\times}_1)\otimes_{U_\hbar^{[]}}\Psi_\hbar^{-1}(\CA^{\BC^\times}_2)\right)\\
   =\Ext^\bullet_{D^{G^\vee}(U_\hbar^{[]})}\left(U_\hbar^{[]},
   \Phi_\hbar(\CA^{\BC^\times}_1)\otimes_{U_\hbar^{[]}}\fC_{G^\vee}\Phi_\hbar(\CA^{\BC^\times}_2)\right),
 \end{multline*}
  (the last equality is~Lemma~\ref{3.14}(b)).
Now it is easy to see that $\Ext^\bullet_{D^{G^\vee}(U_\hbar^{[]})}\left(U_\hbar^{[]},
   \Phi_\hbar(\CA^{\BC^\times}_1)\otimes_{U_\hbar^{[]}}\fC_{G^\vee}\Phi_\hbar(\CA^{\BC^\times}_2)\right)$
   is the hamiltonian reduction $(\Phi_\hbar(\CA^{\BC^\times}_1)\otimes\fC_{G^\vee}
   \Phi_\hbar(\CA^{\BC^\times}_2))/\!\!/\!\!/\Delta_{G^\vee}$ of
   $\Phi_\hbar(\CA^{\BC^\times}_1)\otimes\fC_{G^\vee}\Phi_\hbar(\CA^{\BC^\times}_2)$ with respect
   to the diagonal action of $G^\vee$. Finally, according to~Lemma~\ref{3.14},
   $H^\bullet_{G_\CO\rtimes\BC^\times}(\Gr_G,\CA^{\BC^\times}_R\otimes^!\CA^{\BC^\times}_{1,2})=
   \Phi_\hbar(\CA^{\BC^\times}_{1,2})$.
\end{proof}

\subsection{Higgs branches of Sicilian theories}
\label{sicilian}
We denote $i_\Delta^!(\CA_R^{\boxtimes b})$ by $\CA^b\in D_{G_\CO}(\Gr_G)$.
It is equipped with an action of $b$ copies of $G^\vee\ltimes U_\hbar^{[]}$.
We denote by $\CB\in D_{G_\CO}(\Gr_G)$ the quantum hamiltonian reduction
of $\CA^2$ by the diagonal action $G^\vee$. We expect that $\CB$ is isomorphic to
$\pi_*\bfomega_{\CR_{G,\fg}}[-2\dim\fg_\CO]$ (see~Section~\ref{ring objects}
and~Example~\ref{steinberg}).
Finally, we set
$\CB^g:=i_\Delta^!(\CB^{\boxtimes g})$. Then $\CA^b\otimes^!\CB^g$ is a commutative
ring object of $D_{G_\CO}(\Gr_G)$, and its equivariant cohomology is a
commutative ring. We denote by $W^{g,b}_G$ its spectrum
$\on{Spec}H^\bullet_{G_\CO}(\Gr_G,\CA^b\otimes^!\CB^g)$. It is a Poisson
variety equipped with an action of $(G^\vee)^b$, the conjectural Higgs
branch of a Sicilian theory.

Recall that according to~\cite{mt}, there is
a conjectural functor from the category of 2-bordisms to a category HS of
holomorphic symplectic varieties with Hamiltonian group actions. The objects
of HS are complex algebraic semisimple groups. A morphism from $G$ to $G'$
is a holomorphic symplectic variety $X$ with a $\BC^\times$-action scaling
the symplectic form with weight 2, together with hamiltonian $G\times G'$-action
commuting with the $\BC^\times$-action. For $X\in\on{Mor}(G',G),\ Y\in\on{Mor}(G,G'')$,
the composition $Y\circ X\in\on{Mor}(G',G'')$ is given by the symplectic reduction
of $Y\times X$ by the diagonal $G$-action. The identity morphism in $\on{Mor}(G,G)$
is the cotangent bundle $T^*G$ with the left and right action of $G$.

To a complex semisimple group $G$ and a Riemann surface with boundary,
physicists associate a $3d$ Sicilian theory and consider its Higgs branch.
It depends only on the topology of the Riemann surface, and gives a functor
as above. Such a functor satisfying most of expected properties was constructed
recently in~\cite{gk}. It follows from~Proposition~\ref{fusion} that the above
$W^{g,b}_G$ is associated to the group $G^\vee$ and Riemann surface of genus $g$
with $b$ boundary components. It is also proved in~\cite{bfn4} that
$W^{0,3}_{\on{PGL}(2)}\simeq\BC^2\otimes\BC^2\otimes\BC^2$, and
$W^{0,3}_{\on{PGL}(3)}$ is the minimal nilpotent orbit of $E_6$,
while $W^{1,1}_{\on{PGL}(3)}$ is the subregular nilpotent orbit of $G_2$,
as expected by physicists.

\section{Coulomb branches of $3d$ quiver gauge theories}\label{quiver}

\subsection{Quiver gauge theories}
\label{quiver gauge}
Let $Q$ be a quiver with $Q_0$ the set of vertices, and $Q_1$ the set of arrows.
An arrow $e\in Q_1$ goes from its tail $t(e)\in Q_0$ to its head $h(e)\in Q_0$.
We choose a
$Q_0$-graded vector spaces $V:=\bigoplus_{j\in Q_0}V_j$ and
$W:=\bigoplus_{j\in Q_0}W_j$. We set $\sG=\GL(V):=\prod_{j\in Q_0}\GL(V_j)$.
We choose a second grading
$W=\bigoplus_{s=1}^NW^{(s)}$ compatible with the $Q_0$-grading of $W$.
We set $\sG_F$ to be a Levi subgroup $\prod_{s=1}^N\prod_{j\in Q_0}\GL(W^{(s)}_j)$
of $\GL(W)$, and $\tilde\sG:=\sG\times\sG_F$.

\medskip
\noindent
{\bf Remark.} $\sG$ will be the gauge group in this section. We denote it by $\sG$ since we want to use the notation $G$ for some other group.

\medskip
\noindent
Finally, we define a central
subgroup $\sZ\subset\sG_F$ as follows: $\sZ:=\prod_{s=1}^N\Delta_{\BC^\times}^{(s)}
\subset\prod_{s=1}^N\prod_{j\in Q_0}\GL(W^{(s)}_j)$, where
$\BC^\times\cong\Delta_{\BC^\times}^{(s)}\subset\prod_{j\in Q_0}\GL(W^{(s)}_j)$ is the
diagonally embedded subgroup of scalar matrices.
The reductive group $\tilde\sG$ acts naturally on $\bN:=
\bigoplus_{e\in Q_1}\Hom(V_{t(e)},V_{h(e)})\oplus\bigoplus_{j\in Q_0}\Hom(W_j,V_j)$.

The Higgs branch of the corresponding quiver gauge theory is the Nakajima
quiver variety $\CM_H(\sG,\bN)=\fM(V,W)$. We are interested in the Coulomb
branch $\CM_C(\sG,\bN)$.

\subsection{Generalized slices in an affine Grassmannian}
\label{general}
Recall the slices $\ol\CW{}_\mu^\lambda$ defined in~Section~\ref{overview}
for domimant $\mu$.
For arbitrary $\mu$ we consider the moduli space $\ol\CW{}_\mu^\lambda$ of
the following data:

\textup{(a)} A $G$-bundle $\CP$ on $\BP^1$.

\textup{(b)} A trivialization $\sigma\colon \CP_{\on{triv}}|_{\BP^1\setminus\{0\}}
\iso\CP|_{\BP^1\setminus\{0\}}$ having a pole of degree $\leq\lambda$ at $0\in\BP^1$
(that is defining a point of $\ol\Gr{}_G^\lambda$).

\textup{(c)} A $B$-structure $\phi$ on $\CP$ of degree $w_0\mu$ with the
fiber $B_-\subset G$ at $\infty\in\BP^1$ (with respect to the trivialization
$\sigma$ of $\CP$ at $\infty\in\BP^1$). Here $G\supset B_-\supset T$ is the
Borel subgroup opposite to $B$, and $w_0\in W$ is the longest element.

This construction goes back to~\cite{fm}. The space $\ol\CW{}_\mu^\lambda$ is
nonempty iff $\mu\leq\lambda$. In this case it is
an irreducible affine normal Cohen-Macaulay variety of dimension
$\langle2\rho^{\!\scriptscriptstyle\vee},\lambda-\mu\rangle$, see~\cite{bfn3}.
In case $\mu$ is dominant, the two definitions of $\ol\CW{}_\mu^\lambda$
agree. At the other extreme, if $\lambda=0$, then $\ol\CW{}_{-\alpha}^0$ is
nothing but the open zastava space $\oZ^{-w_0\alpha}$.
The $T$-fixed point set $(\ol\CW{}_\mu^\lambda)^T$ is nonempty iff the weight
space $V^\lambda_\mu$ is not 0; in this case $(\ol\CW{}_\mu^\lambda)^T$
consists of a single point denoted $\mu$.

\subsection{Beilinson-Drinfeld slices}
\label{beilinson}
Let $\ul\lambda=(\lambda_1,\ldots,\lambda_N)$ be a collection of dominant
coweights of $G$. We consider the moduli space $\ul{\ol\CW}{}_\mu^{\ul\lambda}$ of
the following data:

\textup{(a)} A collection of points $(z_1,\ldots,z_N)\in\BA^N$ on the
affine line $\BA^1\subset\BP^1$.

\textup{(b)} A $G$-bundle $\CP$ on $\BP^1$.

\textup{(c)} A trivialization
$\sigma\colon \CP_{\on{triv}}|_{\BP^1\setminus\{z_1,\ldots,z_N\}}
\iso\CP|_{\BP^1\setminus\{z_1,\ldots,z_N\}}$
with a pole of degree $\leq\sum_{s=1}^N\lambda_s\cdot z_s$ on the complement.

\textup{(d)} A $B$-structure $\phi$ on $\CP$ of degree $w_0\mu$ with the
fiber $B_-\subset G$ at $\infty\in\BP^1$ (with respect to the trivialization
$\sigma$ of $\CP$ at $\infty\in\BP^1$).

$\ul{\ol\CW}{}_\mu^{\ul\lambda}$ is nonempty iff
$\mu\leq\lambda:=\sum_{s=1}^N\lambda_s$. In this case it is an irreducible
affine normal Cohen-Macaulay variety flat over $\BA^N$ of relative dimension
$\langle2\rho^{\!\scriptscriptstyle\vee},\lambda-\mu\rangle$, see~\cite{bfn3}.
The fiber over $N\cdot0\in\BA^N$ is nothing but $\ol\CW{}_\mu^\lambda$.

\subsection{Convolution diagram over slices}
\label{convolution}
In the setup of~Section~\ref{beilinson} we consider the moduli space
$\ul{\wt\CW}{}_\mu^{\ul\lambda}$ of the following data:

\textup{(a)} A collection of points $(z_1,\ldots,z_N)\in\BA^N$ on the
affine line $\BA^1\subset\BP^1$.

\textup{(b)} A collection of $G$-bundles $(\CP_1,\ldots,\CP_N)$ on $\BP^1$.

\textup{(c)} A collection of isomorphisms
$\sigma_s\colon \CP_{s-1}|_{\BP^1\setminus\{z_s\}}
\iso\CP_s|_{\BP^1\setminus\{z_s\}}$
with a pole of degree $\leq\lambda_s$ at $z_s$. Here $1\leq s\leq N$,
and $\CP_0:=\CP_{\on{triv}}$.

\textup{(d)} A $B$-structure $\phi$ on $\CP_N$ of degree $w_0\mu$ with the
fiber $B_-\subset G$ at $\infty\in\BP^1$ (with respect to the trivialization
$\sigma_N\circ\ldots\circ\sigma_1$ of $\CP_N$ at $\infty\in\BP^1$).

A natural projection
$\varpi\colon \ul{\wt\CW}{}_\mu^{\ul\lambda}\to\ul{\ol\CW}{}_\mu^{\ul\lambda}$
sends $(\CP_1,\ldots,\CP_N,\sigma_1,\ldots,\sigma_N)$ to
$(\CP_N,\sigma_N\circ\ldots\circ\sigma_1)$. We denote
$\varpi^{-1}(\ol\CW{}_\mu^\lambda)$ by $\wt\CW{}_\mu^{\ul\lambda}$. Then
we expect that $\varpi\colon \wt\CW{}_\mu^{\ul\lambda}\to\ol\CW{}_\mu^\lambda$
is stratified semismall.

\subsection{Slices as Coulomb branches}
\label{back}
Let now $G$ be an adjoint simple simply laced algebraic group. We choose
an orientation $\Omega$ of its Dynkin graph (of type $ADE$), and denote by
$I$ its set of vertices. Given an $I$-graded vector space $W$ we encode its
dimension by a dominant coweight
$\lambda:=\sum_{i\in I}\dim(W_i)\omega_i\in\Lambda^+$ of $G$.
Given an $I$-graded vector space $V$ we encode its dimension by a positive
coroot combination $\alpha:=\sum_{i\in I}\dim(V_i)\alpha_i\in\Lambda_+$.
We set $\mu:=\lambda-\alpha\in\Lambda$. Given a direct sum decomposition
$W=\bigoplus_{s=1}^NW^{(s)}$ compatible with the $I$-grading of $W$ as
in~Section~\ref{quiver gauge}, we set
$\lambda_s:=\sum_{i\in I}\dim(W_i^{(s)})\omega_i\in\Lambda^+$, and
finally, $\ul\lambda:=(\lambda_1,\ldots,\lambda_N)$.

Recall the notations of~Section~\ref{flavor}. Since the flavor group
$\sG_F$ is a Levi subgroup of $\GL(W)$, its weight lattice is naturally
identified with $\BZ^{\dim W}$. More precisely, we choose a basis
$w_1,\ldots,w_{\dim W}$ of $W$ such that any $W_i,\ i\in I$, and
$W^{(s)},\ 1\leq s\leq N$, is spanned by a subset of the basis, and we assume
the following monotonicity condition: if for $1\leq a<b<c\leq\dim W$ we have
$w_a,w_b\in W^{(s)}$ for certain $s$, then $w_b\in W^{(s)}$ as well.
We define a strictly convex cone
$\sV=\{(n_1,\ldots,n_{\dim W})\}\subset\Lambda^+_F\subset\BZ^{\dim W}$
by the following conditions: (a) if $w_k\in W^{(s)},\ w_l\in W^{(t)}$,
and $s<t$, then $n_k\geq n_l\geq0$; (b) if $w_k,w_l\in W^{(s)}$, then $n_k=n_l$.
The following theorem is proved in~\cite{bfn3,bfn5} by the fixed point localization and
reduction to calculations in rank 1:
\begin{thm}
  We have isomorphisms
\begin{equation*}
  \ol\CW{}_\mu^\lambda\iso\CM_C(\sG,\bN),\
  \ul{\ol\CW}{}_\mu^{\ul\lambda}\iso\ul\CM{}_C^\sZ(\sG,\bN),\
  \ul{\wt\CW}{}_\mu^{\ul\lambda}\iso\wt{\ul\CM}{}_C^{\sZ,\sV}(\sG,\bN),\
\wt\CW{}_\mu^{\ul\lambda}\iso\wt\CM{}_C^\sV(\sG,\bN).
\end{equation*}
\end{thm}

\subsection{Further examples}
Let now $Q$ be an {\em affine} quiver of type $\tilde{A}\tilde{D}\tilde{E}$;
the framing $W$ is 1-dimensional concentrated at the extending vertex; and
the dimension of $V$ is $d$ times the minimal imaginary coroot $\delta$.
Then it is expected that $\CM_C(\sG,\bN)$ is isomorphic to the Uhlenbeck
(partial) compactification $\CU^d_G(\BA^2)$~\cite{bfg}
of the moduli space of $G$-bundles on $\BP^2$ trivialized
at $\BP^1_\infty$, of second Chern class $d$. This is proved for $G=\on{SL}(N)$
in~\cite{nt}.

Furthermore, let $Q$ be a star-shaped quiver with $b$ legs of length $N$ each, and
with $g$ loop-edges at the central vertex. The framing is trivial, and the dimension
of $V$ along each leg, starting at the outer end, is $1,2,\ldots,N-1,N$ (with
$N$ at the central vertex). Contrary to the general setup in~Section~\ref{quiver gauge},
we define $\sG$ as the quotient of $\GL(V)$ by the diagonal central
subgroup $\BC^\times$ (acting trivially on $\bN$). Then according to~\cite{bfn4},
$\CM_C(\sG,\bN)$ is isomorphic to $W_{\on{PGL}(N)}^{g,b}$ of~Section~\ref{sicilian}.

\section{More physics: topological twists of 3d N=4 QFT and categorical constructions}\label{categorical}

The constructions of this Section are mostly conjectural. The main idea of this Section is given by equation (\ref{cat}) which is due to T.~Dimofte, D.~Gaiotto, J.~Hilburn and P.~Yoo. We discuss some interesting corollaries of this equation.

\subsection{Extended topological field theories} Physical quantum field theories usually depend on a choice of metric on the space-time.
The theory is called topological if all the quantities (e.g.\ corelation functions) are independent of the metric (however, look at the warning at the end of the next subsection).
Mathematically, the axioms of a topological QFT were first formulated by Atiyah (cf.~\cite{atiyah}). Roughly speaking, a topological quantum field theory in dimension $d$ consists of the following data:

(a) A complex number $Z(M^d)$ for every closed $d$-dimensional manifold $M^d$;

(b) A space $Z(M^{d-1})$ for every closed $(d-1)$-dimensional manifold $M^{d-1}$;

(c) A vector in $Z(\partial M)$ for every compact oriented $d$-dimensional manifold $M$ with boundary $\partial M$.

\noindent
These data must satisfy certain list of standard axioms; we refer the reader to~\cite{atiyah} for details.
In addition, one can consider a richer structure called {\em extended topological field theory}. This structure in addition to (a), (b) and (c) as above must associate
$k$-category $Z(M^{d-k-1})$ to a closed manifold $M^{d-k-1}$ of dimension $d-k-1$. It should also associate an object of the $k$-category $Z(\partial M)$ to every compact oriented manifold $M$ of dimension $d-k$; more generally, there is a structure associated with every {\em manifold with corners} of dimension $\leq d$. We refer the reader to \cite{lurie} for details about extended topological field theories.
In the sequel we shall be mostly concerned with the case $d=3$. In this case one is supposed to associate a (usual) category to the circle $S^1$.
Physicists call it {\em the category of line operators}.

\subsection{Topological twists of 3d N=4 theories}
Physical quantum field theories are usually not topological. However, sometimes physicists can produce a universal procedure which associates a  topological field theory to a physical theory with enough super-symmetry. Since in these notes we are not discussing what a quantum field theory really is, we can't discuss what a topological twist really is. Physicists say that any 3d N=4 theory with some mild additional  structure\footnote{The nature of this additional structure will become more clear in Subsection \ref{cotangent}.} must have two topological twists (we'll call them Coulomb and Higgs twists, although physicists often call them $A$ and $B$ twists by analogy with similar construction for 2-dimensional field theories). These twists must be interchanged by the 3d mirror symmetry operation mentioned in Section \ref{naive}.

\subsection{Warning}\label{warning} The twists are topological only in some weak sense. Namely, in principle as was mentioned above in a topological field theory everything (e.g.\ correlators) should be independent of the metric (i.e.\ only depend on the topology of the relevant space-time).  In a weakly topological field theory everything should be metric-independent only locally. This issue will be ignored in this section since we are only going to discuss some pretty robust things but it is actually important if one wants to understand some finer aspects.

\subsection{The category of line operators in a topologically twisted 3d N=4 theory}
To a 3d TFT one should be able to attach a ``category of line operators" (i.e.\ this is the category one attaches to a circle in terms of the previous subsection). Morever, since the circle $S^1$ is the boundary of a canonical 2-dimensional manifold: the 2-dimensional disc, this category should come equipped with a canonical object. In this Section we would like to suggest a construction of these categories together with the above object for a wide class of topologically twisted 3d N=4 theories (we learned the idea of this construction from T.~Dimofte, D.~Gaiotto, J.~Hilburn and P.~Yoo who can actually derive this construction from physical considerations. To the best of our knowledge their paper on the subject is forthcoming).

A priori the above categories of line operators  should be $\ZZ_2$-graded. However, as was mentioned above, in order to define the relevant topological twists one needs to choose some mild additional structure on the theory (we explain this additional structure in series of examples in Subsection \ref{cotangent}). So we are actually going to think about them as
$\ZZ$-graded categories (in fact, as dg-categories). But we should keep in mind that if we choose this additional structure in a different way, then a priori we should get different $\ZZ$-graded categories but with the same underlying $\ZZ_2$-graded categories.

Since a 3d N=4 theory is supposed to have two topological twists which we call Coulomb and Higgs, we shall denote the corresponding categories of line operators  by $\calC_C,\calC_H$. As was mentioned above, filling the circle with a disc should
produce canonical objects $\calF_C,\calF_H$.

\medskip
\noindent
{\bf Remark for an advanced reader.} In principle in a true TQFT the category of line operators should be an $E_2$-category (cf.~\cite{lurie-ek}).
There is a closely related notion of factorizable category (in the $D$-module sense), a.k.a.\ chiral category, cf.~\cite{raskin}. In fact, the categories we are going to construct will be factorizable categories (and the canonical object, corresponding to the 2d disc will be a factorizable object). The fact that we get factorizable categories as opposed to $E_2$-categories is related to the warning in subsection \ref{warning}.

\medskip
\noindent
The relation between these structures and what we have discussed in the previous Sections is that one should have
\begin{equation}\label{cat-coulomb}
\Ext^*(\calF_C,\calF_C)=\CC[\calM_C]
\end{equation}
and
\begin{equation}\label{cat-higgs}
\Ext^*(\calF_H,\calF_H)=\CC[\calM_H].
\end{equation}

\medskip
\noindent
{\bf Remark.} It can be shown that for any factorization category $\calC$ and a factorization object $\calF$ the algebra $\Ext^*(\calF,\calF)$ is graded commutative.

\medskip
\noindent
When we need to emphasize dependence on a theory $\calT$, we shall write $\calC_C(\calT),\calF_C(\calT)$ etc.
The mirror symmetry conjecture then says
\begin{conj}\label{mirror}
The category $\calC_C(\calT)$ is equivalent to $\calC_H(\calT^*)$ (and the same with $C$ and $H$ interchanged).
Under this equivalence the object $\calF_C(\calT)$ goes over to $\calF_H(\calT^*)$.
\end{conj}
\subsection{Generalities on $D$-modules and de Rham pre-stacks}
In what follows we'll need to work with various categories of sheaves on spaces which are little more general than usual schemes or stacks.
Namely, we need to discuss de Rham pre-stacks and various categories of sheaves related to them. Our main reference for the subject is \cite{GaiRo}.

Let $S$ be a smooth scheme of finite type over $\CC$. Then one can define certain pre-stack (i.e.\ a functor from $\CC$-algebras to sets) $S_{dR}$ which is called {\em the de Rham pre-stack of $S$}. Informally it is defined as the quotient of $S$ by infinitesimal automorphisms.
Moreover, this definition can be extended to all schemes, stacks or even dg-stacks of finite type over $\CC$. A key property of $S_{dR}$ is that the category of quasi-coherent sheaves on $S_{dR}$ is the same as the category of $D$-modules on $S$.\footnote{Because we plunge ourselves into world of derived algebraic geometry here, it doesn't make sense to talk about either quasi-coherent sheaves or $D$-modules as an abelian category: only the corresponding derived category makes sense.}
In addition for a  target stack $\calY$ one can consider the mapping space $\Maps(S_{dR},\calY)$. Here are two important examples:

(1) Let $\calY=\AA^1$. Then $\Maps(S_{dR},\calY)$ is the de Rham cohomology of $S$ considered as a dg-scheme.

(2) Let $\calY=\pt/G$ where $G$ is an algebraic group. Then $\Maps(S_{dR},\calY)$ is the stack of $G$-local systems on $S$ (i.e.\ the stack classifying $G$-bundles on $S$ endowed with a flat connection).

\noindent
In the sequel we'll need to apply these constructions to $S$ being either the formal disc $\calD=\Spec(\calO)$ or the punctured disc $\calD^*=\Spec(\calK)$. This is not formally a special case of the above as some completion issues arise if one tries to spell out a careful definition. However, with some extra care all definitions can be extended to this case. This is done in \cite{Gai-Jer}.

\subsection{Construction of the categories in the cotangent case}\label{cotangent}
It is expected that one can attach the above theories and categories to any symplectic dg-stack $\calX$.
It is now easy to spell out the additional structure on the theory that one needs in order to define the two topological twists in terms
of the stack $\calX$. Namely, one needs a $\CC^{\times}$-action on $\calX$ with respect to which the symplectic form $\omega$ has weight 2.

We shall actually assume that $\calX=T^*\calY$ where $\calY$ is a smooth stack; in this case the above $\CC^{\times}$-action is automatic (we can just use the square of the standard $\CC^{\times}$-action on the cotangent fibers).
 We shall denote this theory by $\calT(\calY)$.

%In what follows we set $\calD=\operatorname{Spec}(\calO), \calD^*=\operatorname{Spec}(\calK)$.
The following construction is due to T.~Dimofte, D.~Gaiotto, J.~Hilburn and P.~Yoo (private communication). Namely, let us set
\begin{equation}\label{cat}
\calC_C=\Dmod(\Maps(\calD^*,\calY));\quad \calC_H=\QCoh(\Maps(\calD_{dR}^*,\calY)).
\end{equation}
Let us stress that both $\Dmod$ and $\QCoh$ mean the corresponding derived categories.

Let now $\pi_C: \Maps(\calD_{dR},\calY)\to \Maps(\calD^*_{dR},\calY)$ be the natural map; similarly we define $\pi_H$. Then, we set
\begin{equation}\label{obj}
\calF_H=(\pi_H)_*\CO_{\Maps(\calD_{dR},\calY)};\quad \calF_C=(\pi_C)_!\CO_{\Maps(\calD,\calY)}.
\end{equation}

\subsection{A very important warning}\label{warning}
The above suggestion is probably only an approximation of a true statement. In fact, we believe that the suggestion is fine for $\calC_C$; however, for $\calC_H$ some modifications might be necessary. Let, for example (for simplicity), $\calZ$ be a dg-stack of finite type over $\CC$. Then following \cite{ArGa} in addition to the category $\QCoh(\calZ)$ one can also study the derived category $\IndCoh(\calZ)$ of ind-coherent sheaves on $\calZ$. The two categories coincide when $\calZ$ is a smooth classical (i.e.\ not dg) stack. But for more general $\calZ$ these categories are different. This can be seen as follows: the compact objects of $\IndCoh(\calZ)$ are by definition finite complexes with coherent cohomology, while the compact objects of $\QCoh(\calZ)$ are finite perfect complexes. Moreover, assume that $\calZ$ is locally a complete intersection. Then in \cite{ArGa} the authors define certain stack $\Sing(\calZ)$ endowed with a representable map $\Sing(\calZ)\to \calZ$, which is an isomorphism when $\calZ$ is a smooth classical stack. Moreover, the fibers of this map are vector spaces; in particular, there is a natural $\CC^{\x}$-action on the fibers whose stack of fixed points is naturally identified with $\calZ$. Given a closed conical substack $\calW\subset \Sing(\calZ)$ the authors in \cite{ArGa} define a category $\IndCoh_{\calW}(\calZ)$ of ind-coherent sheaves with singular support in $\calW$. These categories in some sense interpolate between $\QCoh(\calZ)$ and $\IndCoh(\calZ)$: namely, when $\calW=\calZ$ (the zero section of the morphism $\Sing(\calZ)\to \calZ$) we have  $\IndCoh_{\calW}(\calZ)=\QCoh(\calZ)$, and when $\calW=\Sing(\calZ)$
we have $\IndCoh_{\calW}(\calZ)=\IndCoh(\calZ)$.\footnote{The reader should be warned that although we have a natural functor $\IndCoh_{\calW}(\calZ)\to \IndCoh(\calZ)$, this functor is not fully faithful, so $\IndCoh_{\calW}(\calZ)$ is not a full subcategory of $\IndCoh(\calZ)$.}

We think that suggestion (\ref{cat}) is only ``the first approximation" to the right statement. More precisely, we believe that it is literally the right suggestion for the category $\calC_C$, but for the category $\calC_H$ one has to be more careful. We believe that the correct definition of the category $\calC_H$ in the above context should actually be $\IndCoh_{\calW}(\Maps(\calD_{dR}^*,\calY))$ for a particular choice of $\calW$ (very often $\calW$ will actually be the zero section but probably not always); at this moment we don't know how to specify $\calW$ in the above generality.
The purpose of the rest of the Section will be to explain some general picture, so in what follows we are going to ignore this subtlety, i.e.\ we shall proceed with the suggestion $\calC_H=\QCoh(\Maps(\calD_{dR}^*,\calY))$
as stated. But the reader should keep in mind that in certain cases it should be replaced by $\IndCoh_{\calW}(\Maps(\calD_{dR}^*,\calY))$ (this issue will become important when we formulate some rigorous conjectures (cf.\ for example the discussion before Conjecture \ref{grass}).

\subsection{Remarks about rigorous definitions}
Since the above mapping spaces are often genuinely infinite-dimensional, we need to discuss why the above categories make sense.
First, the category of $D$-modules on arbitrary pre-stack is discussed in \cite{Raskin-dmod}.
In fact, in {\em loc.\ cit.} the author defines two versions of this
category, which are denoted by $D^!$ and $D^*$ (these two categories are dual to each other). For the purposes of these notes we need to work with $D^*$: for example since it is this category for which the functor of direct image is well-defined.

The category $\QCoh(\mathcal Z)$ is well-defined for any pre-stack $\mathcal Z$; however in such generality it might be difficult to work with. However, we would like to note that $\Maps(\calD_{dR}^*,\calY)$ is typically a very manageable object. For example, when $\calY$ is a scheme of finite type over $\CC$ it follows from Conjecture \ref{small} below that $\Maps(\calD_{dR}^*,\calY)$ is a dg-scheme of finite type over $\CC$, so QCoh is ``classical" (modulo the fact that we have to work with commutative dg-algebras as opposed to usual commutative algebras). When $\calY$ is a stack, the definition is a bit less explicit; however, we claim that the definition is easy when $\calY$ is of the form $\calS/G$ where $\calS$ is an affine scheme and $G$ is a reductive group. For example, when $\calY=\pt/G$ we have $\Maps(\calD_{dR}^*,\calY)=\text{LocSyS}_G(\calD^*)$: the stack of $G$-local systems (i.e.\ principal $G$-bundles with a connection on $\calD^*)$, and QCoh$(\text{LocSyS}_G(\calD^*))$ is a well-studied object in (local) geometric Langlands correspondence.

Here is another reason why we want $\calY$ to be of the above form.
The map $\pi_H$ is actually always a closed embedding, so we could write $(\pi_H)_!$ instead of $(\pi_H)_*$. On the other hand, the functor $(\pi_C)_!$ is a priori not well defined, at least it is not defined for an arbitrary morphism. However, it is well-defined if the morphism $\pi_C$ is ind-proper. In what follows we shall always assume that the stack $\calY$ is such that it is the case. This condition is not always satisfied but it is also not super-restrictive as follows from the next exercise.

\medskip
\noindent
{\bf Exercise.} (a) Show that $\pi_C$ is a closed embedding if $\calY$ is a scheme.

(b) Show that if $\calY=\calS/G$ where $\calS$ is an affine scheme and $G$ is a reductive algebraic group then the morphism $\pi_C$ is ind-proper.

(c) Show that (b) might become false if we drop either the assumption that
$\calS$ is affine or the assumption that $G$ is reductive.

\medskip
\noindent

We shall  denote the corresponding categories~(\ref{cat}) and
objects~(\ref{obj}) simply by $\calC_C(\calY),\calF_C(\calY)$ etc.
Note that these categories are $\ZZ$-graded. The above arguments then suggest the following
\begin{conj}\label{z2}
Let $\calY,\calY'$ be two stacks such that $T^*\calY$ is isomorphic to $T^*\calY'$ as a symplectic dg-stack.
Then the corresponding $\ZZ_2$-graded versions of $\calC_C(\calY)$ and $\calC_C(\calY')$ are equivalent as $\ZZ_2$-graded factorization categories;
this equivalence sends $\calF_C(\calY)$ to $\calF_C(\calY')$. Similar statement holds for $\calC_H$.
\end{conj}

\subsection{Small loops}
In fact, one can demistify the category $\calC_H(\calY)$ a little bit which makes it quite computable.
First of all, with the correct definition it is easy to see that $\Maps(\calD_{dR},\calY)$ is equivalent to $\calY$ (for any $\calY$).\footnote{Here we see that $\Maps(\calD_{dR},\calY)$ should be defined with some extra care. Namely, if we just used the naive definition then the equivalence $\Maps(\calD_{dR},\calY)\simeq \calY$ would imply that $\calD_{dR}=\on{pt}$ which is far from being the case.}
Now, given $\calY$ let us define another dg-stack $L\calY$ (we shall call it ``small loops" into $\calY$) by setting
$$
L\calY=\calY\underset {\calY\times \calY}\times \calY,
$$
where in the above equation both maps $\calY\to \calY\times \calY$ are equal to the diagonal
map.\footnote{Here we want to stress once again that all fibered products must be understood in the dg-sense!}
We have a natural map $\calY\to L\calY$.
\begin{conj}\label{small}
\begin{enumerate}
\item
Let $\calY$ be a scheme. Then $L\calY$ and $\Maps(\calD^*_{dR},\calY)$ are isomorphic (and this isomorphism is compatible with the map from $\calY=\Maps(\calD_{dR},\calY)$ into both).
\item
Let $\calY$ be a stack. Then the formal neighbourhoods of $\calY=\Maps(\calD_{dR},\calY)$ in $\Maps(\calD^*_{dR},\calY)$ and in $L\calY$ are equivalent.
\end{enumerate}
\end{conj}
The proof of Conjecture \ref{small} will be written in a different publication. In what follows we shall assume Conjecture \ref{small}.

\subsection{Remark}If $\calY$ is a scheme then it is easy to see that both $\Maps(\calD^*_{dR},\calY)$ and $L\calY$ are dg-extensions of $\calY$ (i.e.\ they are dg-schemes whose underlying classical scheme is $\calY$), so if the statement of Conjecture \ref{small} holds on the level of formal neighbourhoods then in fact we have $L\calY=\Maps(\calD^*_{dR},\calY)$. This is not the case for stacks. Namely, let $G$ be a reductive algebraic group and let $\calY=\pt/G$. Then it is easy to see that $\Maps(\calD^*_{dR},\calY)$ is the stack $\on{LocSys}_G(\calD^*)$ of $G$-local systems on $\calD^*$ (i.e.\ principal $G$-bundles on $\calD^*$ with a connection).

\medskip\noindent
{\bf Exercise.} Show that for $\calY=\pt/G$ we have $L\calY=G/\on{Ad}(G)$ (i.e.\ quotient of $G$ by itself with respect to the adjoint action).
Show that $G/\on{Ad}(G)$ is not equivalent to $\on{LocSys}_G(\calD^*)$ but the formal neighbourhoods of $\pt/G$ in both are equivalent (the embedding of $\pt/G$ into $\on{LocSys}_G(\calD^*)$ corresponds to the trivial local system).
\subsection{An example}\label{cat-example}
The significance of Conjecture \ref{small} is that it allows to use the (very explicit) stack $L\calY$ in order to compute the Ext-algebra (\ref{cat-higgs}).

Assume that $\calY$ is a smooth scheme which for simplicity we shall also assume to be affine.
Then $L\calY$ is just the dg-scheme
$\Spec(\Sym_{\CO_{\calY}} T^*\calY[1])$. Since we have
$$
\Ext^*_{\Sym_{\CO_{\calY}} T^*\calY[1])}(\CO_{\calY},\CO_{\calY})=\Sym_{\CO_{\calY}}(T\calY[-2]),
$$
we see that (with grading disregarded) $\CC[\calM_H]=\CC[T^*\calY]$ which is what we should have in this case.
In fact, if we want to rembember the grading we see that the homological grading on the RHS goes to grading coming from dilation of the cotangent fibers on the LHS. Recall that writing $\calX$ as $T^*\calY$ is an additional structure which is precisely the one required in order to make all the categories $\ZZ$-graded (as opposed to $\ZZ_2$-graded; note also that the grading on $\Sym_{\CO_{\calY}}(T\calY[-2])$ is even, so the corresponding $\ZZ_2$-grading is trivial).

Let us now compute $\CC[\calM_C]$ in this case. Since $\calY$ is an affine scheme, it follows that
$\Maps(\calD,\calY)$ is a closed subscheme in the ind-scheme $\Maps(\calD^*,\calY)$, so
$\Ext^*(\calF_C,\calF_C)$ is just equal to the de Rham cohomology of $\Maps(\calD,\calY)$. Since $\calY$ is smooth the (evaluation at $0\in \calD$) map $\Maps(\calD,\calY)\to \calY$ is a fiber bundle whose fibers are infinite-dimensional affine spaces. Thus it induces an isomorphism on
de Rham cohomology. Hence we get $\CC[\calM_C]=H^*(\calY,\CC)=H^*(T^*\calY,\CC)$. So if $\calY$ is connected, we see that $\calM_C$ is a dg-extension of pt; moreover, if $\calY$ is a vector space,
that $\calM_C=\pt$ even as dg-schemes (as was promised in subsection \ref{basic-example}).

\subsection{Gauge theory}
Consider now the the example when $\calY=\bfN/G$, where $G$ is a connected reductive group and $\bfN$ is a representation of $G$.

\medskip
\noindent
{\bf Exercise} Show that in this case the RHS of (\ref{cat-coulomb}) is literally the same as $H^{G_\CO}_\bullet(\CR_{G,\bN})$.

\medskip
So, we see that our categorical point of view recovers the definition of the Coulomb branch we gave before. Let us look at the Higgs branch.
According to Conjecture \ref{small} we need to understand the dg-stack
\begin{equation}\label{n/g}
L(\bfN/G)=\bfN/G\underset{\bfN/G\times \bfN/G}\times \bfN/G.
\end{equation}
Let us actually first assume that $\bfN$ is any smooth variety with a $G$-action. Then it is easy to see that (\ref{n/g}) is a dg-stack which admits the following description. The action of $G$ on $\bfN$ defines a natural map of locally free $\CO_{\bfN}$-modules
$$
\grg\otimes\CO_{\bfN}\to T\bfN.
$$
Consider the dual map
$$
T^*\bfN\to \grg^*\otimes \CO_{\bfN}
$$
and let us regard it as two step complex of coherent sheaves on $\bfN$ where $T^*\bfN$ lives in degree $-1$ and $\grg^*\otimes \CO_{\bfN}$ lives in degree $0$. Let us denote this complex by $K^{\bullet}$. Then $\Sym_{\CO_{\bfN}}(K^{\bullet})$ is a quasi-coherent dg-algebra on $\bfN$.

\medskip
\noindent
{\bf Exercise.} Show that the formal neighbourhood of $\bfN/G$ in $L(\bfN/G)$ is equivalent to the formal neighbourhood of $\bfN/G$ in
 $\Spec(\Sym_{\CO_{\bfN}}(K^{\bullet}))/G$ (note that when $G$ is trivial we just recover $\Spec(\Sym(T^*\bfN)[1]))$ as in the previous subsection).

\medskip
\noindent
It now  follows that the RHS of (\ref{cat-higgs}) in our case becomes equal to the $G$-invariant part of
\begin{equation}\label{bred}
\Ext^*_{\Sym_{\CO_{\bfN}}(K^{\bullet})}(\CO_{\bfN},\CO_{\bfN}).
\end{equation}
Assume now for simplicity that $\bfN$ is affine. Then it is easy to see that (\ref{bred}) is equal to the cohomology of  $\Sym_{\CO_{\bfN}}((K^{\bullet})^*[-1])$.

\medskip
\noindent
{\bf Exercise.} Show that as a $\ZZ_2$-graded algebra $\Sym_{\CO_{\bfN}}((K^{\bullet})^*[-1])$ is quasi-isomorphic to
the algebra of functions on the dg-scheme $\mu^{-1}(0)$ where $\mu:T^*\bfN\to \grg^*$ is the moment map.

\medskip
\noindent
The exercise implies that the RHS of (\ref{cat-higgs}) is isomorphic to the algebra of functions on the dg-stack $\mu^{-1}(0)/G$.
\subsection{Mirror symmetry in the toric case}
Let us assume that we are in the situation of subsection \ref{toric-gauge}.
We set $\calY=\CC^n/T, \calY^*=\CC^n/T_F^{\vee}$.
Combining (\ref{cat}) and (\ref{obj}) with Conjecture \ref{mirror} we already obtain a bunch of non-trivial statements. Namely, we arrive at the following
\begin{conj}\label{toric-mirror}
For the above choice of $\calY$ and $\calY^*$ we have equivalences of (factorization) categories
$$
\begin{aligned}
\Dmod(\Maps(\calD^*,\calY))\simeq \QCoh(\Maps(\calD^*_{dR},\calY^*))\\
\Dmod(\Maps(\calD^*,\calY^*))\simeq \QCoh(\Maps(\calD^*_{dR},\calY)).
\end{aligned}
$$
\end{conj}
A proof of this conjecture is the subject of a current work in progress of the first named author with Dennis Gaitsgory.
Let us discuss the simplest example (which is already quite non-trivial).

Let us take $n=1$ and let $T$ be trivial. In other words we get $\calY=\AA^1$. Then $\calY^*=\AA^1/\GG_m$.
So, let us look closely at what Conjecture \ref{toric-mirror} says in this case.

First, $\Maps(\calD^*_{dR},\calY)=\Maps(\calD^*_{dR},\AA^1)=H^*_{dR}(\calD^*)=\AA^1\times \AA^1[-1]$.
In other words, $\Maps(\calD^*_{dR},\calY)=\Spec(\CC[x,\eps])$ where $\deg(x)=0,\deg(\eps)=-1$ (we consider it as a dg-algebra with trivial differential). The category $\calC_C(\calY)$ is then just the derived category of dg-modules over this algebra.
More precisely, it is the QCoh version of this derived category --- we again refer the reader to Chapter II of \cite{GaiRo}.
The object $\calF_C$ corresponds to the dg-module $\CC[x]$ (on which $\eps$ acts trivially) in degree $0$.

We claim that in this case
the category $\QCoh(\Maps(\calD^*_{dR},\AA^1))$ is equivalent to $D$-mod$(\Maps(\calD^*,\AA^1))$ even as a $\ZZ$-graded category.
We are not in a position to give a rigorous proof here, since for this we'll need to spell out careful definitions of both categories, and that
goes beyond the scope of these notes. Let us give some examples of objects which go to one another under the above equivalence.
First, the object $\calF_C\in D$-mod$(\Maps(\calD^*,\AA^1))$ is described as follows. Let $i_n:\calO\to \calK$ be the embedding which sends
$f$ to $z^n f$ (here $n\in \ZZ$). Then we have
$$
\calF_C(\AA^1/\GG_m)=\bigoplus\limits_{n\in \ZZ} (i_n)_*\CO.
$$

\medskip
\noindent
{\bf Warning.} To understand this object carefully one really needs to spell out the definition. Let us mention the problem one has to fight with. It is intuitively clear that we have a $\ZZ$-action on $\calK$ such that $n\in \ZZ$ sends $f(z)$ to $z^n f(z)$. On the other hand, assume that $n\geq 0$. Then $z^n\calO$ has codimension $n$ in $\calO$ (although one is obtained from the other by means of the $\ZZ$-action). This problem is in fact not as serious as it might seem at the first glance -- it just shows that the actual definition of $D$-modules on $\calK$ (or even on $\calO$) must take into account certain homological shifts.

\medskip
\noindent
Having the above warning in mind, it is easy to see that $\Ext^*(\calF_C,\calF_C)=\CC[x,y]$ where $\deg(x)=0,\deg(y)=2$. On the other hand, we also
have
$$
\Ext^*_{\CC[x,\eps]}(\CC[x],\CC[x])=\CC[x,y],
$$
which matches our expectations.

Here is another example. Consider the module $\CC$ over $\CC[x,\eps]$ (i.e.\ we think of it as a dg-module concentrated in degree 0, on which $x$ acts by 1  and $\eps$ acts by $0$). Then under the above equivalence it goes to the $D$-module $\delta$ of delta-functions at $0\in \calK$ (considered as a $\calK^*$-equivariant $D$-module). Note that the $\calK^*$-equivariant Ext from $\delta$ to itself is the same as $H^*_{\calK^*}(\pt,\CC)$. Now, homotopically $\calK^*$ is equivalent to $\CC^{\times}\times \ZZ$ and we have
$$
H^*_{\CC^{\times}\times \ZZ}(\pt,\CC)=\CC[y,\theta]\quad \text{where $\deg(y)=2,\deg(\theta)=1$}.
$$
On the other hand the same (dg) algebra $\CC[y,\theta]$ is equal to $\Ext^*_{\CC[x,\eps]}(\CC,\CC)$.

\subsection{The theory $\calT[G]$}

Here is another expectation. Let $D(\Gr_G)^{\Hecke}$ denote the derived category of Hecke eigen-modules on $\Gr_G$, i.e.\ $D$-modules which are also right modules for the algebra $\calA_R$.
\begin{conj}\label{tg}
The category $\calC_C(\calT[G])$ is the category $D(\Gr_G)^{\Hecke}$ and $\calF_C=\calA_R$.
\end{conj}
Let us combine it with (\ref{cat}). In the case when $G=\GL(n)$ the theory $\calT[G]$ does in fact come from a smooth stack $\calY$;
here
\begin{equation}\label{ygln}
\calY=(\prod\limits_{i=1}^{n-1} \Hom(\CC^i,\CC^{i+1}))/\prod\limits_{i=1}^{n-1} \GL(i)
\end{equation}
(note that $\calY$ still has an action of $\GL(n)$).
So, from (\ref{cat}) we get another construction of $\calC_C$ which should be equivalent to the one from \ref{tg}.

It is in fact easy to construct a functor in one direction. Namely, let $\calC$ be a category with a $D$-module action of some group $G$; let also $\calF$ be a $G_{\calO}$-equivariant object.
Then $\calF$ defines a functor $\calC\to D$-mod$(\Gr_G)$. Moreover, this functor sends $\calF$ to a ring object $\calA_{\calF}$ and the above functor can be upgraded to a functor from $\calC$ to
$\calA_{\calF}$-modules in $D$-mod$(\Gr_G)$. Namely, this functor sends every $\calG$ to the $D$-module on $\Gr_G$ whose $!$-stalk at some $g$ is equal to $\RHom(\calF^g,\calG)$. In our case we take $\calC$ to be the category of $D$-modules on $\Maps(\calD^*,\calY)$ and take
$\calF=\calF_C$. Then the above functor sends $\calF$ to $\calA_R$ (this is essentially proved
in~\cite{bfn4}).

Note that for $G=\GL(n)$ the theory $\calT[G]$ is supposed to be self-dual (with respect
to mirror symmetry procedure). Hence it follows that in this case
the category $\calC_C$ should be equivalent to $\calC_H$.
Therefore, it is natural to expect that the category $\QCoh(\Maps(\calD^*_{dR},\calY))$ is equivalent to $D(\Gr_{\GL(n)})^{\Hecke}$.
However, we expect that it is actually wrong as stated -- the reason is the warning from subsection \ref{warning}. However, we do believe in the following 
\begin{conj}
Let $\calY$ be as in (\ref{ygln}). Then
the category $\IndCoh(\Maps(\calD^*_{dR},\calY))$ is equivalent to $D(\Gr_{\GL(n)})^{\Hecke}$.
\end{conj}
Here is (an equivalent) variant of this conjecture. Note that the action of the group $\GL(n)$ on $\calY$ gives rise to an action of the same group on $\Maps(\calD^*_{dR}, \calY)$. Hence we can consider the category $\on{QCoh}(\Maps(\calD^*_{dR}, \calY)/\GL(n))$. This category admits a natural action of the tensor category $\on{Rep}(\GL(n))$. Note that the geometric Satake equivalence also gives rise to an action of  $\on{Rep}(\GL(n))$ on $D$-mod$(\Gr_{\GL(n)})$ (the action is by convolution with spherical $D$-modules on the right).
\begin{conj}\label{grass}
The categories $\on{IndCoh}(\Maps(\calD^*_{dR}, \calY)/\GL(n))$ and $\Dmod(\Gr_{\GL(n)})$ are equivalent as module categories over $\on{Rep}(\GL(n))$.
\end{conj}
In the paper \cite{bf-gl2} we prove a weaker version of Conjecture \ref{grass} for $\GL(2)$ (in particular, the version of Conjecture \ref{grass} proved in  \cite{bf-gl2} is sufficient in order to explain why we need IndCoh and not QCoh in the formulation).

\subsection{$G$-symmetry and gauging}
Let us now address the following question. Let $\calT$ be a theory acted on by a (reductive) algebraic group $G$. What kind of structures does this action imply in terms of the categories $\calC_H,\calC_C$?

To answer this question, we need to recall two general notions. First, given a category $\calC$ and a group ind-scheme $\calG$ there is a notion  {\em strong} or {\em infinitesimally trivial} $\calG$-action on $\calC$ (cf.~\cite{Gai-act}). The main example of such an action is as follows: given a pre-stack $\calS$ with a $\calG$-action, the group $\calG$ acts strongly on the (derived) category of $D$-modules on $\calS$. If one replaces $D$-modules by quasi-coherent sheaves, one gets the notion of {\em weak} $G$-action on a category $\calC$. Given a category $\calC$ with a strong $\calG$-action one can define the category of {\em strongly equivariant objects in $\calC$}
(cf.~\cite[page~4]{Gai-act}); we shall denote this category by $\calC^{\calG}$.

On the other hand, for a stack $\calZ$ and a (dg-)category $\calC$ there is a notion of ``$\calC$ living over $\calZ$" (cf.~\cite{Gai-cat}). This simply means that the category $\on{QCoh}(\calZ)$ (which is a tensor category) acts on $\calC$. Given a geometric point $z$ of $\calZ$ we can consider the fiber $\calC_z$ of $\calC$ at $z$. This category always has a weak action of the group Aut$_z$ of automorphisms of the point $z$.

Now we can formulate an (approximate) answer to the above question. Namely, we expect that a $G$-action on $\calT$ should produce the following structures:

\medskip
(1) A category $\calC_H(G,\calT)$ which lives over $\on{LocSys}_G(\calD^*)$ endowed with an equivalence
$$
\calC_H(G,\calT)_{\mathbf{Triv}}\simeq \calC_H(T).
\footnote{Again, the reader should keep in mind subsection \ref{warning}.}
$$
Here $\mathbf{Triv}$ stands for the trivial local system.

(2) A strong $G(\calK)=\Maps(\calD^*,G)$-action on the category $\calC_C(\calT)$.

\medskip
\noindent
Note that $G$ is the group of automorphisms of the trivial local system. Hence~(1) implies that a $G$-action on $\calT$ yields a weak action of $G$ on $\calC_H(\calT)$.

\medskip
\noindent
{\bf Exercise.} Show that this action extends to a weak action of $LG$ (which is a dg-extension of $G$) on $\calC$.

\medskip
\noindent
The reader must be warned that a weak action of $G$ or even of $LG$ on $\calC_H(\calT)$ is a very small amount of data: for example, it is not sufficient in order to reconstruct $\calC_H(G,\calT)$.

Recall now that if a group $G$ acts on a theory $\calT$ then we can form the corresponding gauge theory $\calT/G$. Then we expect that
\begin{equation}\label{gauge-cat}
\calC_H(\calT/G)=\calC(G,\calT);\quad \calC_C(\calT/G)=\calC_C(\calT)^{G(\calK)}.
\end{equation}

Let us now go back to the case $\calT=\calT(\calY)$. In this case an action of $G$ on $\calY$ yields an action of $G$ on $\calT(\calY)$.
In this case we expect that $\calT(\calY)/G=\calT(\calY/G)$.
Let us discuss the compatibility of this statement with above categorical structures. First, an action of $G$ on $\calY$ gives rise to an action of $G(\calK)$ on
$\Maps(\calD^*,\calY)$, hence a strong action on $D$-mod$(\Maps(\calD^*,\calY))$. Moreover,
$D$-mod$(\Maps(\calD^*,\calY))^{G(\calK)}=D$-mod$(\Maps(\calD^*,\calY/G))$ which is compatible with the 2nd equation of (\ref{gauge-cat}).
On the other hand, it is easy to see that the category $\on{QCoh}(\Maps(\calD^*_{dR}, \calY/G))$ lives over $\on{QCoh}(\LS_G(\calD^*))$ and its fiber over $\mathbf{Triv}$ is $\on{QCoh}(\Maps(\calD^*_{dR}, \calY))$ which is compatible with the first equation of
(\ref{gauge-cat}).

\subsection{$S$-duality and local geometric Langlands}\label{glanglands}This subsection is a somewhat side topic: here we would like to mention a possible connection of the above discussion with (conjectural) local geometric Langlands correspondence. A reader who is not interested in the subject is welcome to skip this subsection.

The local geometric Langlands duality predicts the existence of an equivalence $\bfL_G$ between the ($\infty$-)category of (dg-)categories with strong $G(\calK)$-action and the ($\infty$-)category of (dg-)categories over $\QCoh(\LS_{G^{\vee}}(\calD^*))$ (as was already mentioned earlier in these notes we are going to ignore higher categorical structures, which are in fact necessary in order to discuss these things rigorously).\footnote{It is known that this is only an approximate conjecture. The correct conjecture (due to A.~Arinkin) requires a (rather tricky) modification of the notion category over $\LS_G(\calD^*)$ (which again has to do with the difference between QCoh and IndCoh).}

Let us now recall that given a theory $\calT$ with a $G$-action one expects the existence of the $S$-dual theory $\calT^{\vee}$ with a $G^{\vee}$-action. Thus we see that we get a category $\calC_C(\calT^{\vee})$ with a strong $G^{\vee}(\calK)$-action and a category
$\calC_H(\calT^{\vee}/G^{\vee})$ which lives over $\LS_{G^{\vee}}(\calD^*)$.

\begin{conj}\label{langlands}
We have natural equivalences
$$
\bfL_G(\calC_C(\calT))\simeq \calC_H(\calT^{\vee}/G^{\vee});\quad \bfL_{G^{\vee}}(\calC_C(\calT^{\vee}))\simeq \calC_H(\calT/G).
$$
\end{conj}

Recall now formula (\ref{gw}):
$$
\calT^{\vee}=((\calT\times \calT[G])/ G)^*.
$$
In particular, we can apply it to $\calT$ being the trivial theory; in this case we get that the group $G^{\vee}$ should act on the theory
$(\calT[G]/G)^*$. Let $\calC_G=\calC_C((\calT[G]/G)^*)$. Then this category should have a strong action of $G^{\vee}(\calK)$. On the other hand, $\calC_G=\calC_H(\calT[G]/G)$, so in addition it should live over $\LS_G(\calD^*)$ (these two structures should commute in the obvious way). We expect that $\calC_G$ is the {\em universal Langlands category} for $G^{\vee}$, i.e.\ that for any other category $\calC$ with $G^{\vee}(\calK)$-action we have
$$
\bfL_{G^{\vee}}(\calC)=\calC\underset {G^{\vee}(\calK)}\otimes \calC_G.\footnote{Such a tensor product does make sense as long as we live in the world of dg-categories.}
$$
In particular, if $G=\GL(n)$ then we see that the universal Langlands category $\calC_{\GL(n)}$ is expected to be equivalent to $\QCoh(\Maps(\calD^*_{dR},\calY/G))$ where $\calY$ is given by (\ref{ygln}). Note that in this realization the fact that this category lives over $\LS_{\GL(n)}(\calD^*)$ is clear, but the action of $\GL(n,\calK)$ is absolutely not obvious: we don't know how to construct it.

Again, it must be noted that the notion of universal Langlands category is not precise since as was mentioned above the correct formulation of the local geometric Langlands conjecture involves a modification of the notion of category over $\LS_G(\calD^*)$. But at least we believe that the above description of the universal Langlands category is true as stated over the locus of irreducible local systems.
\subsection{Quantization}
Let us now discuss the categorical structures which give rise to the quantizations (and thus to Poisson structures) of the algebras $\CC[\calM_H]$ and $\CC[\calM_C]$. Let us first look at the
latter one. The space $\Maps(\calD^*,\calY)$ has a natural action of the multiplicative group $\GG_m$ (which acts on $\calD^*$ by loop rotation). Thus the category $\calC_C(\calY)$ admits a natural deformation:
the category $D$-$\on{mod}_{\GG_m}(\Maps(\calD^*,\calY))$ of $\GG_m$-equivariant $D$-modules. The object $\calF_C$ deforms naturally to an object of $D$-$\on{mod}_{\GG_m}(\Maps(\calD^*,\calY))$ and thus we can set
$$
\CC_{\hbar}(\calM_C)=\Ext^*_{\Dmod_{\GG_m}(\Maps(\calD^*,\calY))}(\calF_C,\calF_C).
$$
Here $\hbar$ as before is a generator of $H^*_{\GG_m}(\pt,\CC)$.

What about the quantization of $\calM_H$? As before we need to look for a one-parameter deformation of the pair $(\calC_H,\calF_H)$.
Here again the action of the multiplicative group $\GG_m$ on $\calD$ and on $\calD^*$ gives rise to an action of $\GG_m$ on the category
$\calC_H(\calY)$; we claim that this action is strong (this is related to the fact that we work with maps from $\calD^*_{dR}$ rather than with maps from $\calD^*$).
Thus it makes sense to consider the category of strongly equivariant objects in $\calC_H(\calY)$
(cf.\ again \cite[page~4]{Gai-act}). The Ext-algebra of (the natural  analog of) the object $\calF_H$ in this category is again a non-commutative algebra over $\CC[\hbar]$ which is  a quantization of $\CC[\calM_H]$.

Note that since in both cases we use the action of the multiplicative group on $\calD^*$, it follows that the deformed categories are no longer factorisation categories, so the corresponding Ext-algebras no longer have factorisation structure. This is why they have a chance to become non-commutative (at least the Remark after~(\ref{cat-higgs}) does not apply here).

\subsection{Holomorphic-topological twist}
We have learned the main ideas of this subsection from K.~Costello.
So far we discussed the two topological twists of a given theory completely independently of each other. However, in fact in physics both the $C$-twist and the $H$-twist appear as one-parametric families of equivalent twists. In addition, both families have {\em the same} limiting point, where the theory is no longer topological (it becomes {\em holomorphic-topological}, cf.~\cite{ACMV}; roughly speaking it means that, for example, for a 3-manifold $M$ the partition function $Z(M)$ is well-defined if one fixes some additional structure on $M$ which locally makes it look like a product of a complex curve $\Sigma$ and a 1-manifold $I$). The category of line operators in the holomorphic-topological theory is still well-defined. As a result we come to the following conclusion:

\medskip
\noindent
{\bf Conclusion:}
There should exists a factorisation category $\calC$ with an object $\calF$ and two $\ZZ$-gradings such that

(1) The two $\ZZ$-gradings yield the same $\ZZ_2$-grading.

(2) The pair $(\calC_C,\calF_C)$ is a deformation of the pair $(\calC,\calF)$. This deformation preserves the 1st grading on $\calC$.

(3) The pair $(\calC_H,\calF_H)$ is a deformation of the pair $(\calC,\calF)$. This deformation preserves the 2nd grading on $\calC$.

\medskip
Let us describe a suggestion for the category $\calC(\calY)$ and the object $\calF(\calY)$.
We would like to set
$$
\calC(\calY)=\on{QCoh}(T^*\Maps(\calD^*,\calY)).
$$
Here there are some technical problems: the stack $T^*\Maps(\calD^*,\calY)$ is very essentially infinite-dimensional, so studying quasi-coherent sheaves on it is more difficult than before.
Let us assume that it is possible though and let us discuss (1)-(3) in this case.
First, we need two gradings. The first grading is simply the homological grading on QCoh.
The second grading is the combination of the homological grading and the grading coming from $\CC^{\times}$-action on $T^*\Maps(\calD^*,\calY)$ (which dilates the cotangent fibers) multiplied by two (so the two grading manifestly yield the same $\ZZ_2$-grading).

Now the category of $D$-modules on $\Maps(\calD^*,\calY)$ is clearly a deformation of $\on{QCoh}(T^*\Maps(\calD^*,\calY))$. On the other hand, it is less clear how to deform the category $\on{QCoh}(T^*\Maps(\calD^*,\calY))$ to $\on{QCoh}(\Maps(\calD_{dR}^*,\calY))$. We plan to address these issues in a future publication.

\end{document}